\documentclass[11pt]{amsart}

\usepackage{lmodern}
\usepackage{microtype}
\usepackage{aliascnt} 
\usepackage[pagebackref]{hyperref}
\usepackage[nameinlink,capitalise]{cleveref}
\usepackage{mathrsfs}
\usepackage{amsrefs}
\usepackage{amsthm}
\usepackage{enumerate}
\usepackage{graphicx}
\usepackage{scrextend}
\usepackage{amsfonts, amsmath, amssymb,comment}  
\usepackage{times}
\usepackage{mathdots}%
\usepackage{nccmath}
\usepackage{mathtools}
\usepackage{ulem}
\usepackage{enumitem}
\usepackage{multicol}
		{\end{pmatrix}\end{medsize}}%
\usepackage{dsfont,bbm}
\usepackage{rotating}
\usepackage{lscape}
\usepackage{url}
\usepackage{thmtools, thm-restate}
\usepackage{blkarray}
\usepackage[a4paper,margin=2.5cm]{geometry}
\usepackage{multirow}

\allowdisplaybreaks

\usepackage{pifont}
\usepackage{tikz}
\usetikzlibrary{topaths}
\tikzstyle{every picture}+=[remember picture,inner xsep=0,inner ysep=0.25ex]
\usetikzlibrary{calc}

\DeclareFontFamily{U}{mathx}{\hyphenchar\font45 }
\DeclareFontShape{U}{mathx}{m}{n}{
	<5><6><7><8><9><10><10.95><12><14.4><17.28><20.74><24.88> mathx10
}{}
\DeclareSymbolFont{mathx}{U}{mathx}{m}{n}
\DeclareMathAccent{\widecheck}{0}{mathx}{"71}

\usepackage{kbordermatrix}
\renewcommand{\kbldelim}{(}
\renewcommand{\kbrdelim}{)}
\renewcommand{\kbrowstyle}{\displaystyle}
\renewcommand{\kbcolstyle}{\displaystyle}

\makeatletter
\newcommand*{\sublabel}[1]{%
	\let\old@currentlabel\@currentlabel%
	\renewcommand{\@currentlabel}{\theenumii}%
	\label{#1}%
	\let\@currentlabel\old@currentlabel%
}
\makeatother
\allowdisplaybreaks

\graphicspath{ {./figs/} }

\DeclareMathOperator{\Span}{span}

\DeclareMathOperator{\Rank}{rank}

\DeclareMathOperator{\Card}{card}

\DeclareMathOperator{\Dep}{Dep}
\DeclareMathOperator{\Depn}{DepN}
\DeclareMathOperator{\Ker}{Ker}
\DeclareMathOperator{\mult}{mult}
\DeclareMathOperator{\col}{col}
\DeclareMathOperator{\row}{row}

\DeclareMathOperator{\coef}{coef}
\DeclareMathOperator{\rank}{rank}

\makeatletter
\def\widebreve{\mathpalette\wide@breve}
\def\wide@breve#1#2{\sbox\z@{$#1#2$}%
	\mathop{\vbox{\m@th\ialign{##\crcr
				\Kern0.08em\brevefill#1{0.8\wd\z@}\crcr\noalign{\nointerlineskip}%
				$\hss#1#2\hss$\crcr}}}\limits}
\def\brevefill#1#2{$\m@th\sbox\tw@{$#1($}%
	\hss\resizebox{#2}{\wd\tw@}{\rotatebox[origin=c]{90}{\upshape(}}\hss$}
\makeatletter

\newcommand{\RR}{\mathbb R}

\newcommand{\NN}{\mathbb N}

\newcommand{\cT}{\mathcal T}

\newcommand{\Sym}{\mathbb{S}}

\newcommand{\cS}{\mathcal{S}}
\newcommand{\cE}{\mathcal{E}}

\newcommand{\cH}{\mathcal H}

\newcommand{\cZ}{\mathcal Z}
\newcommand{\cC}{\mathcal C}

\newcommand{\benu}{\begin{enumerate}}
	\newcommand{\eenu}{\end{enumerate}}
\newcommand{\bop}{\begin{opomba}}
	\newcommand{\eop}{\end{opomba}}

\newcommand{\Bor}{\mathrm{Bor}}

\newcommand{\tr}{\mathrm{tr}}
\newcommand{\supp}{\mathrm{supp}}

\newcommand{\CB}[1]{\begin{color}{blue}#1\end{color}}
\newcommand{\CR}[1]{\begin{color}{red}#1\end{color}}

\newtheorem{theorem}{Theorem}[section]
\newtheorem{corollary}[theorem]{Corollary}
\newtheorem{lemma}[theorem]{Lemma}
\newtheorem{proposition}[theorem]{Proposition}

\theoremstyle{definition}

\definecolor{green-new}{rgb}{0.0, 0.5, 0.0}
\newcommand{\CG}[1]{\begin{color}{green-new}#1\end{color}}
\definecolor{cyan}{rgb}{0.0, 0.8, 1.0}

\theoremstyle{definition}

\newtheorem{remark}[theorem]{Remark}

\numberwithin{equation}{section}

\begin{document}
	

	\title[]
	{The matricial univariate rational truncated moment problem}
	
	\author[A. Zalar]{Alja\v z Zalar${}^{1}$}
	\address{Alja\v z Zalar, 
        University of Ljubljana, 
		Faculty of Computer and Information Science \& 
		Faculty of Mathematics and Physics, \&
		Institute of Mathematics, Physics and Mechanics, Ljubljana, Slovenia.}
	\email{aljaz.zalar@fri.uni-lj.si}
	\thanks{${}^1$Supported by the ARIS (Slovenian Research and Innovation Agency)
		research core funding No.\ P1-0288 and grants No.\ J1-50002, J1-60011, 70017.}
	
	\author[I. Zobovi\v c]{Igor Zobovi\v c${}^{2}$}
	\address{Igor Zobovi\v c, 
		Institute of Mathematics, Physics and Mechanics, Ljubljana, Slovenia.}
	\email{igor.zobovic@imfm.si}
	\thanks{${}^2$Supported by the ARIS (Slovenian Research and Innovation Agency)
		research core funding No.\ P1-0288.}

	\hypersetup{ pdftitle={The Matricial Univariate Rational Truncated Moment Problem}, pdfauthor={Aljaz Zalar and Igor Zobovic}, pdfsubject={Truncated matricial rational moment problems}, pdfkeywords={rational moment problem, matrix-valued measure, truncated moment problem, Gaussian quadrature} }
    
	\begin{abstract}
        We solve the matricial univariate rational truncated moment problem on the real line and on a half-line. We give explicit necessary and sufficient conditions for the existence of a positive matrix-valued representing measure and show that every solvable problem admits a finitely atomic representing measure whose total atomic multiplicity equals the rank of the associated moment matrix. The rational problem is reduced to an ordinary matricial moment problem with the additional requirement that the representing measure avoid the real poles of the rational data. The main new ingredient is a simultaneous prescribed-node result: the smallest attainable multiplicities at finitely many prescribed points can be realized by a single minimal representing measure. Our proofs are constructive and use flat extensions, block-column relations, and localizing conditions.
		\looseness=-1
	\end{abstract}

	\date{\today}
	\maketitle
	
	\section{Introduction}
	\label{introduction}  

Let $i,j\in\mathbb{Z}$ with $i\leq j$ and $\ell\in \NN$, {where $\NN := \{1, 2, \ldots\}$}. We use the notation
\[
  [i;j]\coloneqq\{i,i+1,\ldots,j\}
  \qquad \text{and}\qquad
  [\ell]\coloneqq[1;\ell]=\{1,2,\ldots,\ell\}, \quad [0] := \emptyset.
\]
For $k\in\mathbb{N}\cup\{0\}$ and $p\in\mathbb{N}$, let
$\mathbb{R}[x]_{\leq k}$ denote the vector space of real univariate
polynomials of degree at most $k$, and let $\Sym_p(\mathbb{R})$ denote
the space of real symmetric $p\times p$ matrices.

Let $K\subseteq\mathbb{R}$ be closed. Fix pairwise distinct real
numbers
\[
  t_1,\ldots,t_{k_1}\in\mathbb{R},
  \qquad
  k_1\in\mathbb{N}\cup\{0\},
\]
and pairwise distinct real numbers
\[
  u_1,\ldots,u_{k_2}\in\RR,
  \quad
  v_1,\ldots,v_{k_2}\in(0,\infty),
  \qquad
  k_2\in\mathbb{N}\cup\{0\}.
\]
Let
\[
  e_0\in \NN\cup\{0\},\; e_1,\ldots,e_{k_1}\in\mathbb{N}
  \qquad\text{and}\qquad
  f_1,\ldots,f_{k_2}\in\mathbb{N},
\]
and define
\begin{equation}\label{eq:def-of-q}
  q(x)
  \coloneqq
  \prod_{j=1}^{k_1}(x-t_j)^{2e_j}
  \prod_{j=1}^{k_2}((x-u_j)^2+v_j)^{f_j}.
\end{equation}
Set
\[
  2n
  \coloneqq
  \sum_{j=0}^{k_1}2e_j
  +
  \sum_{j=1}^{k_2}2f_j,
\]
and consider the finite-dimensional vector space of rational functions
\begin{equation}\label{eq:def-of-rational-space}
  \mathcal{R}^{(2n)}
  \coloneqq
  \left\{
    \frac{g}{q}:
    g\in\mathbb{R}[x]_{\leq 2n}
  \right\}.
\end{equation}

The \textbf{matricial univariate rational $K$--truncated moment
problem}
asks for a
characterization of the linear mappings
\[
  \mathcal{L}\colon
  \mathcal{R}^{(2n)}\longrightarrow\Sym_p(\mathbb{R})
\]
for which there exists a positive $\Sym_p(\mathbb{R})$-valued measure
$\mu$, supported on $K$, such that
\[
  \mathcal{L}(R)
  =
  \int_K R\,d\mu
  \qquad
  \text{for every }R\in\mathcal{R}^{(2n)}.
\]
The notion of a positive matrix-valued measure is recalled in
Subsection~\ref{subsec:matrix-measure}. Any measure $\mu$ satisfying
the preceding conditions is called a
\textbf{$K$--representing measure for $\mathcal{L}$}.

Equivalently, the matricial rational moment problem can be formulated by prescribing the integrals of monomials and rational functions of the form $x^i$, $\frac{1}{(x-t_j)^\alpha}$, $\frac{1}{((x-u_j)^2+v_j)^\beta}$, and $\frac{x}{((x-u_j)^2+v_j)^\beta}$, where the ranges of the indices are determined by the multiplicities appearing in~\eqref{eq:def-of-q}. This explicit formulation is the same in the scalar and matrix-valued settings, since it follows solely from the partial-fraction decomposition of the rational functions involved; see~\cite{NZ25} for the full formulation. 

The equivalent formulation in terms of prescribed integrals connects the matricial rational moment problem directly with the classical theory of univariate rational moment problems. In the scalar case, rational moment problems on bounded and unbounded intervals have a long history. Their full versions were studied extensively by Jones, Nj{\aa}stad and Thron \cite{JTW80,JT81,JNT84,Nja85,NT86,Nja87,Nja88}, primarily by means of orthogonal and quasi-orthogonal rational functions. Rational truncated moment problems were subsequently considered from the viewpoint of positivity and duality; see~\cite{Cha94} and the subsequent work~\cite{NZ25}, which also corrects some results from~\cite{Cha94}. The approach in~\cite{NZ25} relies, in the case of the real line, on the classical solution of the truncated Hamburger moment problem due to Curto and Fialkow~\cite{CF91}. For general closed semialgebraic subsets of the real line, the nonsingular case in~\cite{NZ25} essentially uses a result of {di Dio and Schm\"udgen \cite[Proposition 2 and Corollary 6]{DS18}} (or ~\cite[Theorem~1.30]{Sch17}). In the singular case, the approach relies on the truncated analogue of the Riesz--Haviland theorem~\cite{CF08}, together with the descriptions of univariate polynomials that are nonnegative on such sets due to Kuhlmann, Marshall, and Schwartz~\cite{KM02,KMS05}.

The present paper makes two main contributions. First, using the characterization in \cite{ZZ+} of the smallest attainable multiplicity at a prescribed point, we show that the smallest attainable multiplicities at any finite collection of prescribed points can be realized simultaneously by a single minimal matrix-valued representing measure. Second, by applying this result to the real poles of the common denominator, we obtain explicit necessary and sufficient conditions for the rational truncated moment problem on $\mathbb{R}$ and $[0,\infty)$.  We also aim to apply the resulting theory to bivariate matricial truncated moment problems on suitable algebraic curves, thereby extending results known in the scalar setting \cite{CF02,CF04,CF05}. Following the reduction methods developed for scalar problems in \cite{Zal21,Zal22a,Zal22b,Zal23,YZ24,YZ26}, and analogously to those used for tracial moment problems in \cite{BZ18,BZ21}, such problems can be reduced to univariate rational moment problems. The theory developed here therefore provides the machinery needed to address their matrix-valued counterparts. The study of bivariate matrix-valued truncated moment problems was initiated by Kimsey and Trachana in \cite{KT22}, where the quadratic and cubic cases were investigated.

A basic idea in the study of rational moment problems is to convert them into ordinary polynomial moment problems. In our setting, the rational functions have the common denominator $q$ defined in~\eqref{eq:def-of-q}, whose real zeros have even multiplicity.
Given a linear mapping $\mathcal{L}$ on the corresponding
finite-dimensional space of rational functions, we associate with it a
linear mapping $L$ on the space of polynomials of the appropriate
degree, defined by
$
  L(f)\coloneqq\mathcal{L}\left(\frac{f}{q}\right).
$
A representing matrix measure for the rational problem corresponds to a representing matrix measure for $L$ that assigns zero mass to every real zero of $q$. Once such a measure has been found, the representing measure for the original rational data is recovered by multiplying it by $q$. Thus, the rational moment problem reduces to an ordinary truncated matricial moment problem with prescribed support constraints. 
Related truncated matrix moment problems with a prescribed gap in the support were studied by Zagorodnyuk \cite{Zag15} using generalized-resolvent methods.
In the scalar setting, this correspondence was used in~\cite{NZ25} to solve the truncated rational moment problem for unbounded closed sets, as well as to complete the solution in the compact case by adding the missing support-avoidance conditions. Questions concerning the existence of quadrature rules with few nodes have also been studied using optimization and moment-theoretic methods; see~\cite{RS18,RT26,BKM26}. Of particular relevance to the present setting is the recent work of Bechere, Kuhlmann, and Mourrain~\cite{BKM26} on scalar quadrature formulas whose nodes avoid a prescribed finite set.

Viewing the rational moment problem as an ordinary matricial truncated
moment problem whose representing measures must avoid the real poles
connects it naturally with the theory of matricial Gaussian quadrature
rules with prescribed atoms. Indeed, excluding a real pole $\lambda$
amounts to finding a representing measure for the associated polynomial
data whose mass at $\lambda$ is zero. Results on minimal representing
measures with prescribed atoms can then be used to construct such a
measure. More generally, one may prescribe both an atom and the
multiplicity of its mass. In the nonsingular case, where the truncated
moment matrix is positive definite, this problem was solved
in~\cite{ZZ25}; the corresponding scalar result was established earlier in~\cite{BKRSV20}. The
singular matricial case was treated in~\cite{ZZ+}, where the column
relations of the moment matrix and their propagation must also be taken
into account. The scalar moment-theoretic approach of~\cite{NZ+}
provides an alternative proof of the one-node result and extends it to
several prescribed nodes. Matricial Gaussian quadrature rules have also
been studied through orthogonal matrix polynomials and the zeros of the
associated matrix polynomial; see~\cite{DLR96,DD02,DS03}. In these
approaches, the atoms and matricial masses are determined after a
suitable odd moment, or equivalently an appropriate extension of the
original moment data, has been fixed. The perspective adopted
in~\cite{ZZ25,ZZ+} is different: rather than fixing the extension in
advance, one constructs an extension whose minimal representing measure
has the prescribed behavior at a given point. A related problem was studied in \cite{FKM24}, where the set of all possible masses at a fixed point, over all representing measures for a given truncated matricial moment sequence, was described and the maximal possible mass was identified. In the multivariate setting, possible atoms of representing measures were characterized in \cite{MS23+}, while possible masses at a prescribed point were investigated in \cite{MS23++}. The focus here is instead on minimal representing measures and on simultaneously attaining the smallest possible multiplicities at finitely many prescribed points. The resulting conditions
are expressed in terms of ranks, kernels, and column relations of moment
and localizing moment matrices. The present rational problem follows the same perspective, with the real poles playing the role of points at which the mass of the representing measure must vanish. 
The main new ingredient needed for the rational problem is the simultaneous nature of this construction. The main technical result shows that the individual lower bounds for the multiplicities at these points can be attained simultaneously by a single minimal representing matrix measure. This simultaneous attainment is then combined with the Hamburger or Stieltjes support conditions to solve the rational problem.\\

For a real $m\times n$ matrix $A$, let 
$\mathcal{C}(A) \coloneqq \{Ax:x\in\mathbb{R}^n\} \subseteq\mathbb{R}^m$ denote its column space. For matrices $B_1,\ldots,B_\ell$ with the same number of columns, write \[ \col(B_i)_{i\in[\ell]} \coloneqq \begin{pmatrix} B_1\\ \vdots\\ B_\ell \end{pmatrix}. \]
    The main result of this paper is the following.

\begin{theorem}\label{thm:main-1}
  Let
  $K=\mathbb{R}$ or $K=[0,\infty)$,
  let
  $
    \Lambda\coloneqq\{t_1,\ldots,t_{k_1}\}\subseteq\mathbb{R}
  $
  be the set of real zeros of the polynomial $q$ defined in~\eqref{eq:def-of-q}, and let
  \[
    \mathcal{L}\colon\mathcal{R}^{(2n)}
    \longrightarrow\Sym_p(\mathbb{R})
  \]
  be a linear mapping, where $\mathcal{R}^{(2n)}$ is defined in~\eqref{eq:def-of-rational-space}. Define the associated linear mapping
  \begin{equation}\label{Riesz-functional}
    L\colon\mathbb{R}[x]_{\leq 2n}
    \longrightarrow\Sym_p(\mathbb{R}),
    \qquad
    L(f)\coloneqq\mathcal{L}(fq^{-1}).
  \end{equation}
  Let
  $S_i\coloneqq L(x^i)$,
    $i\in[0;2n]$,
  and let
  \[
    M(n)
    \coloneqq
    \bigl(S_{i+j-2}\bigr)_{i,j=1}^{n+1}
    =
    \bigl(L(x^{i+j-2})\bigr)_{i,j=1}^{n+1}
  \]
  be the truncated moment matrix of $L$.

  For each $r\in[k_1]$, define the shifted moment matrix
  \[
    \mathcal{M}^{(r)}
    \coloneqq
    \bigl(L((x-t_r)^{i+j-2})\bigr)_{i,j=1}^{n+1},
  \]
  and denote its $j$-th column by $\mathbf{y}_j^{(r)}$, $j\in [(n+1)p]$. Set
  \begin{align*}
    A^{(r)}
    \coloneqq
    \Bigl\{
      i\in[p]\colon\,
      &\text{there exists }k\in[n]\text{ such that}\\
      &\mathbf{y}_{kp+i}^{(r)}
      \in
      \Span\!\left\{
        \mathbf{y}_{p+1}^{(r)},\ldots,
        \mathbf{y}_{kp+i-1}^{(r)}
      \right\},\\
      &\mathbf{y}_{(k-1)p+i}^{(r)}
      \notin
      \Span\!\left\{
        \mathbf{y}_1^{(r)},\ldots,
        \mathbf{y}_{(k-1)p+i-1}^{(r)}
      \right\}
    \Bigr\},
  \end{align*}
  {where $\Span(\emptyset) = \{0\}$.}
    Equivalently, \(A^{(r)}=A_{\cT_{t_r}}\) in the notation of Subsection
    \ref{subsec:column-dependencies}.
  Then the following statements are equivalent:
  \begin{enumerate}
    \item\label{thm:main-1-i}
    The mapping $\mathcal{L}$ admits a $K$--representing measure.

    \item\label{thm:main-1-ii}
    The mapping $\mathcal{L}$ admits a finitely atomic
    $K$--representing measure.

    \item\label{thm:main-1-iii}
    The mapping $\mathcal{L}$ admits a $(\rank M(n))$--atomic
    $K$--representing measure.

    \item\label{thm:main-1-iv}
    The following conditions hold:
    \begin{enumerate}
      \item\label{thm:main-1-iv-a}
      The moment matrix $M(n)$ is positive semidefinite.

      \item\label{thm:main-1-iv-b}
      If $K=\RR$, then
      $
        \mathcal{C}\!\left(
          \col(S_{n+i})_{i\in[n]}
        \right)
        \subseteq
        \mathcal{C}\!\left(M(n-1)\right).
      $

      \item\label{thm:main-1-iv-d}
      If $K=[0,\infty)$, then
      $
        \mathcal{H}_x
        \coloneqq
        \bigl(S_{i+j-1}\bigr)_{i,j=1}^{n}
      $
      is positive semidefinite and
      $$
        \mathcal{C}\!\left(
          \col(S_{n+i})_{i\in[n]}
        \right)
        \subseteq
        \mathcal{C}\!\left(\mathcal{H}_x\right).
      $$
      \item\label{thm:main-1-iv-c} 
      $
        A^{(r)}=\emptyset
      $
      for every $r\in[k_1]$.
    \end{enumerate}
  \end{enumerate}
\end{theorem}

\begin{remark}
    The set $A^{(r)}$ identifies the coordinate indices whose
    first dependence occurs at the same block level with and without the
    zeroth block column. Its cardinality is the lower bound for the
    multiplicity at $t_r$ established in Corollary~\ref{co:theorem-subsequent-moments-finitely-atomic} (where $t = t_r$).
\end{remark}


\subsection*{Reader's guide} 
Section~\ref{sec:prel} introduces the notation and preliminary results used throughout the paper. In particular, we recall the basic notions concerning matrix-valued measures, Riesz mappings, moment matrices, block column relations, changes of basis, and the truncated matricial Hamburger and Stieltjes moment problems.
Section~\ref{sec:multiplicity-prescribed-atom} derives a formula for the multiplicity of a prescribed atom in a given minimal representing measure and establishes the consequences needed in the proofs of the main results. Section~\ref{sec:solution-to-the-K-MRTMP} treats representing measures on $\mathbb{R}$. There we prove that the minimal possible multiplicities at finitely many prescribed points can be attained simultaneously, and use this result to prove Theorem~\ref{thm:main-1} for $K=\mathbb{R}$. Section~\ref{sec:solution-to-the-K-MRTMP-v2} develops the corresponding results for representing measures supported on $[0,\infty)$ and completes the proof of Theorem~\ref{thm:main-1} in the half-line case. Finally, in Section \ref{sec:example} we demonstrate the statement of Theorem~\ref{thm:main-1} on a numerical example.

	\section{Preliminaries}
	\label{sec:prel}

\subsection{Notation}
Let $m,m_1,m_2\in\mathbb{N}$. We denote by
$M_{m_1\times m_2}(\mathbb{R})$ the space of real
$m_1\times m_2$ matrices and write
$
  M_m(\mathbb{R})\coloneqq M_{m\times m}(\mathbb{R})
$
for short. For $A\in M_{m_1\times m_2}(\mathbb{R})$, we denote the
linear span of its columns, called the \textbf{column space} of $A$,
by $\mathcal{C}(A)$.

We denote by $I_m$ the $m\times m$ identity matrix and by
$\mathbf{0}_{m_1\times m_2}$ the $m_1\times m_2$ zero matrix. We also
write
$
  \mathbf{0}_m\coloneqq\mathbf{0}_{m\times m}.
$
Finally, $M_m(\mathbb{R}[x])$ denotes the space of $m\times m$ matrices
with entries in $\mathbb{R}[x]$. Its elements are called
\textbf{matrix polynomials}.

Let $p\in\mathbb{N}$. For $A\in\Sym_p(\mathbb{R})$, the notation
$A\succeq0$ (respectively, $A\succ0$) means that $A$ is positive
semidefinite (psd) (respectively, positive definite (pd)). We denote
the cone of positive semidefinite matrices in $\Sym_p(\mathbb{R})$ by
$\Sym_p^{\succeq0}(\mathbb{R})$.

For a polynomial $f\in\mathbb{R}[x]$, we denote its zero set by
\[
  \mathcal{Z}(f)
  \coloneqq
  \{x\in\mathbb{R}:f(x)=0\}.
\]
If $a\in\mathcal{Z}(f)$ and $f \neq 0$, we denote the multiplicity of $a$ as a zero
of $f$ by $\mult_f a$; {if $a\notin\mathcal{Z}(f)$, we write $\mult_f a = 0$.}

Let $A_k$, $k\in[i;j]$, be matrices of compatible sizes. We define
their block row and block column concatenations, respectively, by
\[
  \row(A_k)_{k\in[i;j]}
  \equiv
  \row(A_i,\ldots,A_j)
  \coloneqq
  \begin{pmatrix}
    A_i & A_{i+1} & \cdots & A_j
  \end{pmatrix}
\]
and 
\[
  \col(A_k)_{k\in[i;j]}
  \equiv
  \col(A_i,\ldots,A_j)
  \coloneqq
  \begin{pmatrix}
    A_i\\
    A_{i+1}\\
    \vdots\\
    A_j
  \end{pmatrix}.
\]
{Throughout, matrix dimensions are allowed to be zero; a matrix with zero rows or zero columns is regarded as an empty matrix. In particular, a block row or block column concatenation over an empty index set is understood as the empty matrix of the compatible size.}

\subsection{Matrix measures}
\label{subsec:matrix-measure}

Let $\Bor(\mathbb{R})$ denote the Borel $\sigma$-algebra on
$\mathbb{R}$. A mapping
\[
  \mu=(\mu_{ij})_{i,j=1}^p
  \colon
  \Bor(\mathbb{R})\to\Sym_p(\mathbb{R})
\]
is called a \textbf{$p\times p$ Borel matrix-valued measure}, or
equivalently a \textbf{positive $\Sym_p(\mathbb{R})$-valued measure},
if the following conditions hold:
\begin{enumerate}
  \item For every $i,j\in[p]$, the set function
  $
    \mu_{ij}\colon\Bor(\mathbb{R})\to\mathbb{R}
  $
  is a signed measure.

  \item For every $\Delta\in\Bor(\mathbb{R})$, one has
  $
    \mu(\Delta)\succeq0.
  $
\end{enumerate}

A positive $\Sym_p(\mathbb{R})$-valued measure $\mu$ is called
\textbf{finitely atomic} if there exists a finite set
$F\subseteq\mathbb{R}$ such that
$
  \mu(\mathbb{R}\setminus F)=\mathbf{0}_p.
$
Equivalently, $\mu$ can be written as
\[
  \mu=\sum_{j=1}^{\ell}\delta_{x_j}A_j
\]
for some $\ell\in\mathbb{N}$, points $x_j\in\mathbb{R}$, and matrices
$A_j\in\Sym_p^{\succeq0}(\mathbb{R})$.

A point $x\in\mathbb{R}$ is called an \textbf{atom} of $\mu$ if
$
  \mu(\{x\})\neq\mathbf{0}_p.
$
If $x_1,\ldots,x_\ell$ are pairwise distinct and
$A_j\neq\mathbf{0}_p$ for every $j\in[\ell]$, then the atoms of $\mu$
are precisely $x_1,\ldots,x_\ell$, and
\[
  A_j=\mu(\{x_j\})
\]
is called the \textbf{mass} of $\mu$ at $x_j$. The rank of $A_j$ is
called the \textbf{multiplicity} of $x_j$ in $\mu$ and is denoted by
\[
  \mult_\mu x_j\coloneqq\rank A_j.
\]
For a point $x\in\mathbb{R}$ that is not an atom of $\mu$, we set
$
  \mult_\mu x\coloneqq0.
$
We say that $\mu$ is \textbf{$r$--atomic} if
\[
  \sum_{j=1}^{\ell}\rank A_j=r.
\]
Thus, throughout this paper, the size of a finitely atomic matrix-valued measure is counted with multiplicity, that is, by the sum of the ranks of its matricial masses, rather than by the number of distinct support points.

Let $\mu$ be a positive $\Sym_p(\mathbb{R})$-valued measure. Its
\textbf{trace measure} is defined by
$
  \tau
  \coloneqq
  \tr(\mu)
  \coloneqq
  \sum_{i=1}^p\mu_{ii}.
$
A polynomial $f\in\mathbb{R}[x]$ is called
\textbf{$\mu$-integrable} if $f\in L^1(\tau)$. In this case, its
matrix integral with respect to $\mu$ is defined entrywise by
\[
  \int_{\mathbb{R}}f\,d\mu
  \coloneqq
  \left(
    \int_{\mathbb{R}}f\,d\mu_{ij}
  \right)_{i,j=1}^p.
\]

\subsection{Representing matrix measures for linear operators}
\label{subsec:representing-matrix-measures}

Assume the notation introduced in Section~\ref{introduction}. Given a
linear mapping
\[
  \mathcal{L}\colon\mathcal{R}^{(2n)}
  \longrightarrow\Sym_p(\mathbb{R}),
\]
we associate with it the linear operator
\[
  L\colon\mathbb{R}[x]_{\leq 2n}
  \longrightarrow\Sym_p(\mathbb{R}),
  \qquad
  L(f)\coloneqq\mathcal{L}(fq^{-1}).
\]
{A positive $\Sym_p(\mathbb{R})$-valued measure
$\mu$, supported on $K$, is called a \textbf{$K$--representing measure for $L$} if
\begin{equation*}
    L(f) =
  \int_K f\,d\mu
  \qquad
  \text{for every }f\in\RR[x]_{\leq 2n}.
\end{equation*}
}

Let $\Lambda\subseteq K$ be a Borel set. Define
\begin{align*}
  \mathcal{M}_{\mathcal{L},K}
  &\coloneqq
  \left\{
    \mu:
    \mu\text{ is a $K$--representing measure for }\mathcal{L}
  \right\},\\
  \mathcal{M}_{L,K}
  &\coloneqq
  \left\{
    \mu:
    \mu\text{ is a $K$--representing measure for }L
  \right\},\\
  \mathcal{M}_{L,K,\Lambda}
  &\coloneqq
  \left\{
    \mu\in\mathcal{M}_{L,K}:
    \mu(\Lambda)=\mathbf{0}_p
  \right\}.
\end{align*}
We denote by
$\mathcal{M}_{\mathcal{L},K}^{(\mathrm{fa})}$,
$\mathcal{M}_{L,K}^{(\mathrm{fa})}$,
$\mathcal{M}_{L,K,\Lambda}^{(\mathrm{fa})}$
the subsets of
$\mathcal{M}_{\mathcal{L},K}$, $\mathcal{M}_{L,K}$, and
$\mathcal{M}_{L,K,\Lambda}$, respectively, consisting of all finitely
atomic measures.

Let $\mu$ be a positive $\Sym_p(\mathbb{R})$-valued measure with
$\supp(\mu)\subseteq K$, and let $f$ be $\mu$-integrable. We denote by
$f\cdot\mu$ the $\Sym_p(\mathbb{R})$-valued measure defined by
\[
  (f\cdot\mu)(E)
  \coloneqq
  \int_E f\,d\mu,
  \qquad E\in\Bor(\mathbb{R}).
\]
If $\mu$ is regarded as a measure on $K$, it is enough to take
$E\in\Bor(K)$.

\begin{proposition}
  \label{pr:proposition-measure-sets}
  Let $q$ be as in~\eqref{eq:def-of-q}, and set
  $
    \Lambda\coloneqq K \cap \{t_1,\ldots,t_{k_1}\}.
  $ 
  Then the following statements hold:
  \begin{enumerate}
    \item
    $
      \mathcal{M}_{\mathcal{L},K}\neq\emptyset
      \Longleftrightarrow
      \mathcal{M}_{L,K,\Lambda}\neq\emptyset.
    $
    \smallskip
    \item
    $\mathcal{M}_{\mathcal{L},K}^{(\mathrm{fa})}\neq\emptyset
      \Longleftrightarrow
      \mathcal{M}_{L,K,\Lambda}^{(\mathrm{fa})}\neq\emptyset.
    $
    \smallskip
    \item\label{phi-bijection}
    The map
    \[
      \Phi\colon
      \mathcal{M}_{L,K,\Lambda}
      \Longrightarrow
      \mathcal{M}_{\mathcal{L},K},
      \qquad
      \Phi(\mu)\coloneqq q\cdot\mu,
    \]
    is a bijection, with inverse
    \[
      \Phi^{-1}(\nu)
      =
      q^{-1}\cdot\nu.
    \]
  \end{enumerate}
\end{proposition}

\begin{proof}
  The proof is identical to that of the scalar case; see
  \cite[Proposition~2.1]{NZ25}.
\end{proof}

\subsection{Riesz mapping and the moment matrix}
\label{subsec:Riesz-mapping-moment-matrix}

Let
$
  L\colon\mathbb{R}[x]_{\leq 2n}\to\Sym_p(\mathbb{R})
$
be a linear operator. The matrices
    $S_i\coloneqq L(x^i)$, $i\in[0;2n]$,
are called the \textbf{matricial moments} of $L$.

Conversely, given a sequence
\begin{equation}\label{def:sequence}
  \mathcal{S}
  \equiv
  \mathcal{S}^{(2n)}
  \coloneqq
  (S_0,S_1,\ldots,S_{2n})
  \in
  \bigl(\Sym_p(\mathbb{R})\bigr)^{2n+1},
\end{equation}
we denote by
\[
  L_{\mathcal{S}}\colon
  \mathbb{R}[x]_{\leq 2n}\to\Sym_p(\mathbb{R})
\]
the unique linear operator satisfying
\[
  L_{\mathcal{S}}(x^i)=S_i,
  \qquad i\in[0;2n],
\]
and call it the \textbf{Riesz mapping associated with
$\mathcal{S}$}. {A $K$--representing measure for $L_{\mathcal{S}}$ will also be called a \textbf{$K$--representing measure for $\mathcal{S}$}.}

For a sequence $\mathcal{S}$ as in~\eqref{def:sequence}, we define its
\textbf{$n$-th truncated moment matrix} by
\begin{equation}\label{def:moment-matrix}
  M(n)
  \equiv
  M_{\mathcal{S}}(n)
  \coloneqq
  \bigl(S_{i+j-2}\bigr)_{i,j=1}^{n+1}
  =
  \kbordermatrix{
    & \mathit{1} & X & X^2 & \cdots & X^n \\
    \mathit{1} & S_0 & S_1 & S_2 & \cdots & S_n \\[0.2em]
    X           & S_1 & S_2 & \iddots & \iddots & S_{n+1} \\[0.2em]
    X^2         & S_2 & \iddots & \iddots & \iddots & \vdots \\[0.2em]
    \vdots      & \vdots & \iddots & \iddots & \iddots & S_{2n-1} \\[0.2em]
    X^n         & S_n & S_{n+1} & \cdots & S_{2n-1} & S_{2n}
  }.
\end{equation}
For each $i\in[0;n]$, we write
\begin{equation}\label{def:columns-mm}
  X_{\mathcal{S}}^i
  \coloneqq
  \col\bigl(S_{i+\ell}\bigr)_{\ell\in[0;n]}
\end{equation}
for the $(i+1)$-st block column of $M_{\mathcal{S}}(n)$. Whenever the
underlying sequence $\mathcal{S}$ is clear from the context, we write
$X^i$ instead of $X_{\mathcal{S}}^i$.

\subsection{Evaluation of a matrix polynomial on $M(n)$}
\label{subsec:evaluation}

Let $\mathcal{S}$ be a sequence as in~\eqref{def:sequence}, and let
$M_{\mathcal{S}}(n)$ be the corresponding truncated moment matrix.
Recall from~\eqref{def:columns-mm} that $X_{\mathcal{S}}^i$ denotes
the $(i+1)$-st block column of $M_{\mathcal{S}}(n)$.

Given a matrix polynomial
\begin{equation}\label{def:matrix-poly}
  P(x)
  =
  \sum_{i=0}^n x^iP_i,
  \qquad
  P_i\in M_p(\mathbb{R}),
\end{equation}
we define its \textbf{evaluation on $M_{\mathcal{S}}(n)$} by
\begin{equation}\label{def:evaluation}
  P(X_{\mathcal{S}})
  \coloneqq
  \sum_{i=0}^n X_{\mathcal{S}}^iP_i
  =
  X_{\mathcal{S}}^0P_0
  +X_{\mathcal{S}}^1P_1
  +\cdots
  +X_{\mathcal{S}}^nP_n
  \in M_{(n+1)p\times p}(\mathbb{R}).
\end{equation}
When the underlying sequence $\mathcal{S}$ is clear from the context,
we write $P(X)$ instead of $P(X_{\mathcal{S}})$.

If
\[
  P(X_{\mathcal{S}})
  =
  \mathbf{0}_{(n+1)p\times p},
\]
then $P(x)$ is called a \textbf{block column relation} of
$M_{\mathcal{S}}(n)$.

For a matrix polynomial $P(x)$ as in~\eqref{def:matrix-poly}, we denote
the block column vector of its matrix coefficients by
\begin{equation}\label{def:coefficients}
  \coef(P)
  \coloneqq
  \col(P_i)_{i\in[0;n]}.
\end{equation}

Finally, the moment matrix $M_{\mathcal{S}}(n)$ is called
\textbf{block--recursively generated} if multiplication by $x$
preserves every block column relation of degree at most $n-1$. More
precisely, if
\[
  P(x)
  =
  \sum_{i=0}^j x^iP_i
  \in M_p(\mathbb{R}[x]),
  \qquad
  j\in[0;n-1],
\]
satisfies
$
  P(X_{\mathcal{S}})
  =
  \mathbf{0}_{(n+1)p\times p},
$
then the matrix polynomial
\[
  (xP)(x)
  \coloneqq
  xP(x)
  =
  \sum_{i=0}^j x^{i+1}P_i
\]
also satisfies
$
  (xP)(X_{\mathcal{S}})
  =
  \mathbf{0}_{(n+1)p\times p}.
$

\subsection{Change of basis in the moment matrix}
\label{subsec:change-of-basis}

Let $\mathcal{S}\equiv\mathcal{S}^{(2n)}$ be as
in~\eqref{def:sequence}, and fix $t\in\mathbb{R}$. Consider the
invertible affine transformation
\[
  \phi_t(x)\coloneqq x-t.
\]
For each $i\in[0;2n]$, define the shifted moment
\begin{equation}\label{def:shifted-moments}
  T_i
  \coloneqq
  L_{\mathcal{S}}\bigl((x-t)^i\bigr)
  =
  \sum_{\ell=0}^i
  \binom{i}{\ell}(-1)^\ell t^\ell S_{i-\ell},
\end{equation}
and set
\begin{equation}\label{def:sequenceT}
  \mathcal{T}_t
  \equiv
  \mathcal{T}_t^{(2n)}
  \coloneqq
  (T_0,T_1,\ldots,T_{2n}).
\end{equation}
Let $M_{\mathcal{T}_t}(n)$ be the moment matrix associated with
$\mathcal{T}_t$, and let $X_{\mathcal{T}_t}^i$ denote its $(i+1)$-st
block column for $i\in[0;n]$.

Let
\[
  J_t\colon\mathbb{R}^{n+1}\to\mathbb{R}^{n+1}
\]
be the linear automorphism induced by the substitution
$x\mapsto x-t$ with respect to the monomial basis
$1,x,\ldots,x^n$. Thus,
\begin{equation}\label{def:J_t}
  J_t\coef(f)
  =
  \coef(f\circ\phi_t)
  =
  \coef\bigl(f(x-t)\bigr),
  \qquad
  f\in\mathbb{R}[x]_{\leq n}.
\end{equation}

\begin{proposition}
  [{\cite[Proposition~2.7]{ZZ+}}]
  \label{prop:change-of-basis}
  With the notation above, the following statements hold:
  \begin{enumerate}
    \item For every matrix polynomial
    $P\in M_p(\mathbb{R}[x])$ of degree at most $n$,
    \[
      (J_t\otimes I_p)\coef(P)
      =
      \coef(P\circ\phi_t)
      =
      \coef\bigl(P(x-t)\bigr).
    \]

    \item The moment matrices $M_{\mathcal{S}}(n)$ and
    $M_{\mathcal{T}_t}(n)$ are related by
    \[
      M_{\mathcal{T}_t}(n)
      =
      (J_t\otimes I_p)^{\mathsf T}
      M_{\mathcal{S}}(n)
      (J_t\otimes I_p).
    \]

    \item The linear maps $J_t$ and $J_t\otimes I_p$ are invertible.

    \item Positive semidefiniteness is preserved under the change of
    basis:
    \[
      M_{\mathcal{T}_t}(n)\succeq0
      \quad\Longleftrightarrow\quad
      M_{\mathcal{S}}(n)\succeq0.
    \]

    \item Rank is preserved under the change of basis:
    \begin{equation}\label{eq:rank-change-of-basis}
      \rank M_{\mathcal{T}_t}(n)
      =
      \rank M_{\mathcal{S}}(n).
    \end{equation}

    \item For every matrix polynomial
    $P\in M_p(\mathbb{R}[x])$ of degree at most $n$,
    \[
      P(X_{\mathcal{T}_t})
      =
      (J_t\otimes I_p)^{\mathsf T}
      (P\circ\phi_t)(X_{\mathcal{S}}).
    \]
  \end{enumerate}
\end{proposition}

\begin{proposition}
  [{\cite[Corollary~2.8]{ZZ+}}]
  \label{cor:block-relations-change-of-basis}
  A matrix polynomial $P\in M_p(\mathbb{R}[x])$ of degree at most $n$
  is a block column relation of $M_{\mathcal{T}_t}(n)$ if and only if
  $P(x-t)$ is a block column relation of $M_{\mathcal{S}}(n)$.
\end{proposition}

\begin{proposition}
  [{\cite[Proposition~2.9]{ZZ+}}]
  \label{prop:block-recursive-generation-change-of-basis}
  The moment matrix $M_{\mathcal{S}}(n)$ is block--recursively generated
  if and only if $M_{\mathcal{T}_t}(n)$ is block--recursively generated.
\end{proposition}

\subsection{Column relations of the moment matrix}
\label{subsec:column-dependencies}

Fix $t\in\mathbb{R}$. Using the notation introduced in
Subsection~\ref{subsec:change-of-basis}, write
\[
  M_{\mathcal{T}_t}(n)
  =\row\bigl(X_{\mathcal{T}_t}^i\bigr)_{i\in[0;n]}.
\]
For each $i\in[0;n]$, denote the columns of the block column
$X_{\mathcal{T}_t}^i$ by $(x_{\mathcal{T}_t})_j^{(i)}$, so that
\begin{equation}\label{columns-of-Yi}
  X_{\mathcal{T}_t}^i
  =
  \begin{pmatrix}
    (x_{\mathcal{T}_t})_1^{(i)}
    &
    (x_{\mathcal{T}_t})_2^{(i)}
    &
    \cdots
    &
    (x_{\mathcal{T}_t})_p^{(i)}
  \end{pmatrix}
  =
  \row\bigl((x_{\mathcal{T}_t})_j^{(i)}\bigr)_{j\in[p]}.
\end{equation}

For $i\in[0;n]$ and $j\in[p]$, define
\begin{align*}
  C_{\mathcal{T}_t}(i,j)
  &\coloneqq
  \row\!\left(
    X_{\mathcal{T}_t}^0,\ldots,X_{\mathcal{T}_t}^{i-1},
    (x_{\mathcal{T}_t})_1^{(i)},\ldots,
    (x_{\mathcal{T}_t})_j^{(i)}
  \right),\\
  C_{\mathcal{T}_t}(i,0)
  &\coloneqq
  \row\!\left(
    X_{\mathcal{T}_t}^0,\ldots,X_{\mathcal{T}_t}^{i-1}
  \right).
\end{align*}
Here, when $i=0$, the block columns preceding
$X_{\mathcal{T}_t}^0$ are absent. Thus,
$C_{\mathcal{T}_t}(0,0)$ is the empty matrix.

We define
\begin{equation}\label{def:Dep}
  \Dep\!\left(X_{\mathcal{T}_t}^i\right)
  \coloneqq
  \left\{
    j\in[p]\colon \rank C_{\mathcal{T}_t}(i,j)
    =
    \rank C_{\mathcal{T}_t}(i,j-1)
  \right\}.
\end{equation}
Equivalently, $j\in\Dep(X_{\mathcal{T}_t}^i)$ if and only if the column
$(x_{\mathcal{T}_t})_j^{(i)}$ belongs to the column space of
$C_{\mathcal{T}_t}(i,j-1)$.

The following result shows that these dependence indices are invariant
under translations of the polynomial variable.

\begin{proposition}\label{pr:lemmaDep}
  For every $t,u\in\mathbb{R}$ and every $i\in[0;n]$, one has
  \[
    \Dep\!\left(X_{\mathcal{T}_t}^i\right)
    =
    \Dep\!\left(X_{\mathcal{T}_u}^i\right).
  \]
\end{proposition}

\begin{proof}
  Fix $i\in[0;n]$. By symmetry, it suffices to prove that
  \[
    \Dep\!\left(X_{\mathcal{T}_t}^i\right)
    \subseteq
    \Dep\!\left(X_{\mathcal{T}_u}^i\right).
  \]
  Let $j\in\Dep(X_{\mathcal{T}_t}^i)$. By the definition of
  $\Dep(X_{\mathcal{T}_t}^i)$, the column
  $(x_{\mathcal{T}_t})_j^{(i)}$ belongs to the column space of
  $C_{\mathcal{T}_t}(i,j-1)$. Hence, there exist vectors
  \[
    s_0,\ldots,s_{i-1}\in\mathbb{R}^p
    \quad\text{and}\quad
    s_{i,(0)}\in\mathbb{R}^{j-1}
  \]
  such that
  \begin{equation}\label{eq:dep-linear-comb}
    (x_{\mathcal{T}_t})_j^{(i)}
    =
    \sum_{k=0}^{i-1}X_{\mathcal{T}_t}^k s_k
    +
    X_{\mathcal{T}_t}^i
    \col\!\left(s_{i,(0)},\mathbf{0}_{(p-j+1) \times 1}\right).
  \end{equation}
  Let $e_j\in\mathbb{R}^p$ be the $j$-th standard basis vector and set
    \begin{equation}\label{eq:def-v} 
         v \coloneqq e_j-\col\!\left(s_{i,(0)},\mathbf{0}_{(p-j+1) \times 1}\right). 
    \end{equation}
  Since $(x_{\mathcal{T}_t})_j^{(i)}
  =X_{\mathcal{T}_t}^i e_j$, equation
  \eqref{eq:dep-linear-comb} is equivalent to
  \begin{equation}\label{eq:block-relation-vector-form}
    X_{\mathcal{T}_t}^i v
    =
    \sum_{k=0}^{i-1}X_{\mathcal{T}_t}^k s_k.
  \end{equation}

  For $k\in[0;i]$, define
  \[
    p_k
    \coloneqq
    \begin{cases}
      -s_k, & k\in[0;i-1],\\
      v,    & k=i,
    \end{cases}
    \qquad
    P_k
    \coloneqq
    \begin{pmatrix}
      p_k & \mathbf{0}_{p\times(p-1)}
    \end{pmatrix}
    \in M_p(\mathbb{R}).
  \]
  Then \eqref{eq:block-relation-vector-form} is equivalent to
  \[
    \sum_{k=0}^i X_{\mathcal{T}_t}^kP_k
    =
    \mathbf{0}_{(n+1)p\times p}.
  \]
  Thus, the matrix polynomial
  \[
    P(x)\coloneqq\sum_{k=0}^i x^kP_k
  \]
  is a block column relation of $M_{\mathcal{T}_t}(n)$.

  The sequence $\mathcal{T}_t$ is obtained from $\mathcal{T}_u$ by the
  change of basis corresponding to
  \[
    x\longmapsto x-(t-u).
  \]
  Therefore, by Proposition~\ref{cor:block-relations-change-of-basis}, the matrix polynomial
  \[
    Q(x)
    \coloneqq
    P\bigl(x-(t-u)\bigr)
    =
    P(x-t+u)
  \]
  is a block column relation of $M_{\mathcal{T}_u}(n)$. Write
  \[
    Q(x)=\sum_{k=0}^i x^kQ_k.
  \]
  Since the last $p-1$ columns of every coefficient $P_k$ are zero,
  the same is true of every $Q_k$. Moreover, translation preserves the
  leading coefficient, so $Q_i=P_i$. Consequently, there exist vectors
  $q_0,\ldots,q_i\in\mathbb{R}^p$ such that
  \[
    Q_k=
    \begin{pmatrix}
      q_k & \mathbf{0}_{p\times(p-1)}
    \end{pmatrix},
    \qquad k\in[0;i],
  \]
  and $q_i=v$.

  Evaluating $Q$ on $M_{\mathcal{T}_u}(n)$ and taking the first column
  gives
  \[
    \sum_{k=0}^iX_{\mathcal{T}_u}^kq_k
    =
    \mathbf{0}_{(n+1)p}.
  \]
  Hence,
  \[
    X_{\mathcal{T}_u}^iv
    =
    \sum_{k=0}^{i-1}X_{\mathcal{T}_u}^k(-q_k).
  \]
Recalling the definition of $v$ in~\eqref{eq:def-v}, we obtain
  \[
    (x_{\mathcal{T}_u})_j^{(i)}
    =
    \sum_{k=0}^{i-1}X_{\mathcal{T}_u}^k(-q_k)
    +
    X_{\mathcal{T}_u}^i
    \col\!\left(s_{i,(0)},\mathbf{0}_{(p-j+1) \times 1}\right).
  \]
  Therefore, $(x_{\mathcal{T}_u})_j^{(i)}$ belongs to the column space
  of $C_{\mathcal{T}_u}(i,j-1)$, and hence
  $
    j\in\Dep\!\left(X_{\mathcal{T}_u}^i\right).
  $
  This proves the desired inclusion. The reverse inclusion follows by
  interchanging $t$ and $u$.
\end{proof}

Similarly, for $i\in[n]$ and $j\in[p]$, define
\begin{align*}
  C_{\mathcal{T}_t}^{(1)}(i,j)
  &\coloneqq
  \row\!\left(
    X_{\mathcal{T}_t}^1,\ldots,X_{\mathcal{T}_t}^{i-1},
    (x_{\mathcal{T}_t})_1^{(i)},\ldots,
    (x_{\mathcal{T}_t})_j^{(i)}
  \right),\\
  C_{\mathcal{T}_t}^{(1)}(i,0)
  &\coloneqq
  \row\!\left(
    X_{\mathcal{T}_t}^1,\ldots,X_{\mathcal{T}_t}^{i-1}
  \right).
\end{align*}
When $i=1$, the block columns preceding $X_{\mathcal{T}_t}^1$ are
absent, so $C_{\mathcal{T}_t}^{(1)}(1,0)$ is the empty matrix. 

\begin{remark}
The difference from $C_{\mathcal{T}_t}(i,j)$ is that
$C_{\mathcal{T}_t}^{(1)}(i,j)$ does not contain the zeroth block
column $X_{\mathcal{T}_t}^0$. Thus, $C_{\mathcal{T}_t}(i,j)$ records
column dependencies relative to all preceding block columns, whereas
$C_{\mathcal{T}_t}^{(1)}(i,j)$ records column dependencies relative
only to the preceding block columns of positive degree. In particular,
$
  \mathcal{C}\!\left(C_{\mathcal{T}_t}^{(1)}(i,j)\right)
  \subseteq
  \mathcal{C}\!\left(C_{\mathcal{T}_t}(i,j)\right).
$
\end{remark}

Set
$
  \Dep_1\!\left(X_{\mathcal{T}_t}^0\right)\coloneqq\emptyset,
$
and, for $i\in[n]$, define
\[
  \Dep_1\!\left(X_{\mathcal{T}_t}^i\right)
  \coloneqq
  \left\{
    j\in[p]\colon \rank C_{\mathcal{T}_t}^{(1)}(i,j)
    =
    \rank C_{\mathcal{T}_t}^{(1)}(i,j-1)
  \right\}.
\]
Equivalently, $j\in\Dep_1(X_{\mathcal{T}_t}^i)$ if and only if
$(x_{\mathcal{T}_t})_j^{(i)}$ belongs to the column space of
$C_{\mathcal{T}_t}^{(1)}(i,j-1)$.

For each $j\in[p]$, define the first-dependence indices
\begin{align*}
  \zeta_{\mathcal{T}_t}(j)
  &\coloneqq
  \min\left\{
    i\in[0;n]:
    j\in\Dep\!\left(X_{\mathcal{T}_t}^i\right)
  \right\},\\
  \zeta_{\mathcal{T}_t}^{(1)}(j)
  &\coloneqq
  \min\left\{
    i\in[n]:
    j\in\Dep_1\!\left(X_{\mathcal{T}_t}^i\right)
  \right\},
\end{align*}
where $\min\emptyset\coloneqq+\infty$. For each $i\in[0;n]$, set
\begin{align*}
  \Depn\!\left(X_{\mathcal{T}_t}^i\right)
  &\coloneqq
  \left\{
    j\in[p]\colon\zeta_{\mathcal{T}_t}(j)=i
  \right\},\\
  \Depn_1\!\left(X_{\mathcal{T}_t}^i\right)
  &\coloneqq
  \left\{
    j\in[p]\colon\zeta_{\mathcal{T}_t}(j)
    =
    \zeta_{\mathcal{T}_t}^{(1)}(j)
    =
    i
  \right\},\\
  \Depn_0\!\left(X_{\mathcal{T}_t}^i\right)
  &\coloneqq
  \Depn\!\left(X_{\mathcal{T}_t}^i\right)
  \setminus
  \Depn_1\!\left(X_{\mathcal{T}_t}^i\right).
\end{align*}
Thus, $\Depn(X_{\mathcal{T}_t}^i)$ consists of those indices that first
become dependent in the block column $X_{\mathcal{T}_t}^i$, while
$\Depn_1(X_{\mathcal{T}_t}^i)$ consists of those indices for which the
first dependence occurs at the same block level even after the block
column $X_{\mathcal{T}_t}^0$ is omitted.

Finally, define
\begin{equation}\label{def:def-of-A_t}
  A_{\mathcal{T}_t}
  \coloneqq
  \bigcup_{i=1}^n
  \Depn_1\!\left(X_{\mathcal{T}_t}^i\right).
\end{equation}
For $t=0$, we write $X^i\coloneqq X_{\mathcal{T}_0}^i$ and omit the
subscript $\mathcal{T}_0$ whenever no confusion can arise.

\medskip


Set
\[
  p_0\coloneqq p-\Card\Dep(X^n).
\]



Let
\begin{equation}\label{eq:defmathcalZ-v2-opt-in-definition}
  \mathcal{D}
  \equiv
  \{z_1,\ldots,z_s\}
  \coloneqq
  \left\{
    i\in[0;n]:
    \Depn\!\left(X_{\mathcal{T}_t}^i\right)\neq\emptyset
  \right\},
\end{equation}
where the elements are listed in increasing order:
\[
  0\leq z_1<\cdots<z_s\leq n.
\]
{If $\mathcal{D} \neq \emptyset$, it follows by Proposition \ref{pr:lemmaDep} and the definition of $z_s$ that}
$
  p_0
  =
  p-\Card\Dep\!\left(X_{\mathcal{T}_t}^{z_s}\right)
$; if $\mathcal{D} = \emptyset$, we have $p_0 = p$.
For $j\in[s]$, define
\begin{equation}\label{def-of-p0-and-pj}
  p_j
  \coloneqq
  \Card\Depn\!\left(
    X_{\mathcal{T}_t}^{z_{s-j+1}}
  \right),
  \qquad
  r_j
  \coloneqq
  \sum_{\ell=1}^j p_\ell.
\end{equation}
{If $\mathcal{D} = \emptyset$, we set $s := 0$ and $z_s \equiv z_0 := 0$ so that the formulas in the statements below remain well-defined.}

\begin{lemma}
  [{\cite[Section~3, Claim~2]{ZZ+}}]
  \label{prop:existence-of-relation-matrices}
  Let $\mathcal{S}\equiv\mathcal{S}^{(2n)}$ be as
  in~\eqref{def:sequence}, and assume that the corresponding moment
  matrix $M(n)$ is block--recursively generated. Fix $t\in\mathbb{R}$,
  {and let $\mathcal{D}$ be as
  in~\eqref{eq:defmathcalZ-v2-opt-in-definition},}
  with
  $p_j$ and $r_j$ as in~\eqref{def-of-p0-and-pj}.

  Then there exist matrices
  \[
    \widetilde{H}_0,\widetilde{H}_1,\ldots,\widetilde{H}_n
    \in M_{p\times(p-p_0)}(\mathbb{R})
  \]
  satisfying the following properties:
  \begin{enumerate}
    \item
    $
      \rank\widetilde{H}_n=p-p_0.
    $
    \smallskip
    \item
    \begin{equation}\label{eq:relation-matrices-column-relation}
      \sum_{i=0}^n
      X_{\mathcal{T}_t}^i\widetilde{H}_i
      =
      \mathbf{0}_{(n+1)p\times(p-p_0)}.
    \end{equation}

    \item If
    $
      \widetilde{s}_i
      \coloneqq
      \dim\left(
        \bigcap_{j=0}^i\Ker\widetilde{H}_j
      \right)$,
      $i\in[0;n-1],$
    then
    \begin{equation}\label{eq:kernel-dimension-lower-bound}
      \sum_{i=0}^{n-1}\widetilde{s}_i
      \geq
      (n-z_s)(p-p_0)
      +
      \sum_{j=1}^{s-1}
        (z_{j+1}-z_j)
        \bigl(p-p_0-r_{s-j}\bigr)
      +
      \Card A_{\mathcal{T}_t}.
    \end{equation}
  \end{enumerate}
\end{lemma}

{
\begin{proposition}
  \label{pr:t-optimal-block-column-relation-existence}
  Let $\mathcal{S}\equiv\mathcal{S}^{(2n)}$ be as
  in~\eqref{def:sequence}, and assume that the corresponding moment
  matrix $M(n)$ is block--recursively generated. Fix $t\in\mathbb{R}$,
  and let $\mathcal{D}$ be as
  in~\eqref{eq:defmathcalZ-v2-opt-in-definition}, with $p_j$ and $r_j$
  as in~\eqref{def-of-p0-and-pj}.

  Then there exists a block column relation $
    H(x)=\sum_{i=0}^n x^iH_i
  $
  of $M_{\cT_t}(n)$ satisfying the following properties:
  \begin{enumerate}
    \item
    $
    \rank H_n=p-p_0
    $.
    \smallskip
    \item
    \begin{equation}\label{de:optimal-block-column-relation-2}
    \sum_{i=0}^n\dim\left(\bigcap_{j=0}^i\Ker H_j\right)
    \geq
    (n+1)p_0
    +(n-z_s)(p-p_0)
    +\sum_{j=1}^{s-1}
      (z_{j+1}-z_j)\bigl(p-p_0-r_{s-j}\bigr)
    +\Card A_{\mathcal{T}_t},
  \end{equation}
  where $A_{\mathcal{T}_t}$ is defined
  in~\eqref{def:def-of-A_t}.
    \smallskip
    \item each coefficient $H_i$ is of the form
    \begin{equation}
    H_i
    =
    \begin{pmatrix}
      \mathbf{0}_{p\times p_0} & \widetilde{H}_i
    \end{pmatrix},
    \qquad i\in[0;n], \qquad {\text{where } \widetilde{H}_i
    \in M_{p\times(p-p_0)}(\mathbb{R})}.
    \end{equation}
  \end{enumerate}
\end{proposition}

\begin{proof}
  By
  Proposition~\ref{prop:block-recursive-generation-change-of-basis},
  the moment matrix $M_{\mathcal{T}_t}(n)$ is block--recursively
  generated. Hence, Lemma~\ref{prop:existence-of-relation-matrices}
  yields matrices
  $
    \widetilde{H}_0,\widetilde{H}_1,\ldots,\widetilde{H}_n
    \in M_{p\times(p-p_0)}(\mathbb{R})
  $
  satisfying
  $
    \rank\widetilde{H}_n=p-p_0,
  $
  together with
  \eqref{eq:relation-matrices-column-relation} and
  \eqref{eq:kernel-dimension-lower-bound}.

  For each $i\in[0;n]$, define
  $
    H_i
    \coloneqq
    \begin{pmatrix}
      \mathbf{0}_{p\times p_0} & \widetilde{H}_i
    \end{pmatrix},
  $
  and set
  $
    H(x)\coloneqq\sum_{i=0}^n x^iH_i.
  $
  By \eqref{eq:relation-matrices-column-relation},
  $H(x)$ is a block column relation of
  $M_{\mathcal{T}_t}(n)$. Moreover,
  \[
    \rank H_n
    =
    \rank\widetilde{H}_n
    =
    p-p_0.
  \]

  For every $i\in[0;n]$, the block form of the matrices $H_j$ gives
  \[
    \bigcap_{j=0}^i\Ker H_j
    =
    \mathbb{R}^{p_0}
    \oplus
    \bigcap_{j=0}^i\Ker\widetilde{H}_j.
  \]
  Consequently,
  \[
    \dim\left(\bigcap_{j=0}^i\Ker H_j\right)
    =
    p_0+
    \dim\left(\bigcap_{j=0}^i\Ker\widetilde{H}_j\right).
  \]
  Since $\widetilde{H}_n$ has full column rank, one has
  $
    \Ker\widetilde{H}_n=\{\mathbf{0}\},
  $
  and hence
  $
    \bigcap_{j=0}^n\Ker\widetilde{H}_j
    =
    \{\mathbf{0}\}.
  $
  It follows that
  \begin{equation}\label{PhiH-equality}
    \sum_{i=0}^n\dim\left(\bigcap_{j=0}^i\Ker H_j\right)
    =
    (n+1)p_0+\sum_{i=0}^{n-1}\dim\left(
        \bigcap_{j=0}^i\Ker\widetilde{H}_j
      \right).
  \end{equation}
  Combining \eqref{PhiH-equality} with \eqref{eq:kernel-dimension-lower-bound} yields \eqref{de:optimal-block-column-relation-2}.
%
\end{proof}
}

\subsection{Support of a representing measure and block column relations} \label{subsec:support-block-column-relations}

The following lemma relates the support and atomic multiplicities of a representing measure to the zeros of the determinant of a block column relation.

\begin{lemma}
  [{\cite[Lemma~2.4]{ZZ+}}]
  \label{lem:support-from-block-column-relation}
  Let $n,p\in\mathbb{N}$, and let
  \[
    \mathcal{S}
    \coloneqq
    (S_0,S_1,\ldots,S_{2n})
    \in\bigl(\Sym_p(\mathbb{R})\bigr)^{2n+1}
  \]
  be a sequence with moment matrix $M(n)$ and an
  $\mathbb{R}$--representing measure $\mu$. Suppose that
  \[
    H(x)=\sum_{i=0}^n x^iH_i\in M_p(\mathbb{R}[x])
  \]
  is a block column relation of $M(n)$ and that $H_n$ is invertible.
  Then:
  \begin{enumerate}
    \item 
        $\supp(\mu)\subseteq\mathcal{Z}(\det H)$.
        \smallskip
    \item
        $\mult_\mu\xi
        \leq
        \mult_{\det H}\xi$
        \quad
        \text{for every }$\xi\in\mathbb{R}$.
  \end{enumerate}
\end{lemma}

\subsection{Determinant of a matrix polynomial} \label{subsec:determinant-matrix-polynomial}

The following result relates the vanishing order of the determinant of a matrix polynomial at a prescribed point to the common kernels of its coefficients.

\begin{lemma}
  [{\cite[Lemma~2.10]{ZZ+}}]
  \label{lem:determinant-matrix-polynomial}
  Let $n,p\in\mathbb{N}$, let $t\in\mathbb{R}$, and let
  \[
    H(x)
    =
    \sum_{i=0}^n (x-t)^iH_i
    \in M_p(\mathbb{R}[x])
  \]
  be a nonzero matrix polynomial. 
  {Let $H_n$ be invertible.
  For $k\in[0;n-1]$} define
  \[
    W_k
    \coloneqq
    \bigcap_{i=0}^k\Ker H_i,
    \qquad
    s_k\coloneqq\dim W_k.
  \]
  Then
  \begin{equation}\label{eq:determinant-matrix-polynomial}
    \det H(x)
    =
      (x-t)^{\sum_{k=0}^{n-1}s_k}g(x),
  \end{equation}
  where $0\neq g(x)\in\mathbb{R}[x]$.
\end{lemma}

\subsection{Lemma on the rank of a matrix}

The following technical lemma will be used in the proof of Theorem \ref{th:mainTheorem_RR}.

\begin{lemma}
[{\cite[Lemma~2.14]{ZZ+}}]
\label{le:rank-completion}
Let $m_1,m_2,c,k,\ell\in\mathbb N$, and let
\[
A_1\in M_{m_1\times k}(\mathbb R),\qquad
A_2\in M_{m_2\times k}(\mathbb R),\qquad
C\in M_{c\times k}(\mathbb R),
\]
and
\[
B_1\in M_{m_1\times\ell}(\mathbb R),\qquad
B_2\in M_{m_2\times\ell}(\mathbb R),\qquad
D\in M_{c\times\ell}(\mathbb R).
\]
Assume that
\begin{enumerate}
\item
\[
\rank A_2
=
\rank
\begin{pmatrix}
A_1\\
A_2
\end{pmatrix}
=
\rank
\begin{pmatrix}
A_1 & B_1\\
A_2 & B_2
\end{pmatrix};
\]
\item
\[
\rank
\begin{pmatrix}
A_2\\
C
\end{pmatrix}
=
\rank
\begin{pmatrix}
A_2 & B_2\\
C & D
\end{pmatrix}.
\]
\end{enumerate}
Then
\[
\rank
\begin{pmatrix}
A_1\\
A_2\\
C
\end{pmatrix}
=
\rank
\begin{pmatrix}
A_1 & B_1\\
A_2 & B_2\\
C & D
\end{pmatrix}.
\]
\end{lemma}


	\subsection{Solutions to the truncated matricial Hamburger and Stieltjes moment problem}

    We conclude this section by recalling the characterizations of the solvability of the truncated matricial Hamburger and Stieltjes moment problems due to Bakonyi and Woerdeman and to Kimsey, respectively. These results also guarantee the existence of finitely atomic representing matrix measures of minimal size.
    
	\begin{theorem}
		[{\cite[Theorem 2.7.6]{BW11}}]
		\label{th:Hamburger-matricial}
		Let $n, p \in \NN$ and let $\mathcal S:=(S_0,S_1,\ldots,S_{2n})\in (\Sym_p(\RR))^{2n+1}$
		be a sequence with a moment matrix $M(n)$.
		Then the following statements are equivalent:
		\begin{enumerate}
			\item 
			There exists a  
			$\RR$--representing matrix measure for $\mathcal S$.
			\item 
			There exists a $(\Rank M(n))$--atomic 
			$\RR$--representing matrix measure for $\mathcal S$.
			\item 
			$M(n)$
			is positive semidefinite
			and 
			$
			\cC\big(
			\col(S_{n+i})_{i\in [n]}
			\big)
			\subseteq
			\cC(M(n-1)).
			$
		\end{enumerate}
	\end{theorem}

	\begin{theorem}
		[{\cite[Theorem 2.18]{Kim22}}]
		\label{th:Stieltjes-matricial}
		Let $n, p \in \NN$ and let $\mathcal S:=(S_0,S_1,\ldots,S_{2n})\in (\Sym_p(\RR))^{2n+1}$
		be a sequence with a moment matrix $M(n)$.
		Then the following statements are equivalent:
		\begin{enumerate}
			\item 
			There exists a  
			$[0,\infty)$--representing matrix measure for $\mathcal S$.
			\item 
			There exists a $(\Rank M(n))$--atomic 
			$[0,\infty)$--representing matrix measure for $\mathcal S$.
			\item 
			$M(n)$ and  
            $\cH_x:=\bigl(S_{i+j-1}\bigr)_{i,j=1}^{n}$
			are positive semidefinite,
			and 
			$
			\cC\big(
			\col(S_{n+i})_{i\in [n]}
			\big)
			\subseteq
			\cC(\cH_x).
			$
		\end{enumerate}
	\end{theorem}

\section{Multiplicity of an Atom in a given Minimal Measure}
\label{sec:multiplicity-prescribed-atom}

The formula \eqref{th:theorem-subsequent-moments-pt2-condition} below is implicit in the proof of \cite[Theorem~1.1]{ZZ+}, where it is used for a single prescribed point. We isolate it here because its precise equality, rather than only the resulting lower bound, will be applied at several prescribed points in Sections \ref{sec:solution-to-the-K-MRTMP} and \ref{sec:solution-to-the-K-MRTMP-v2}. It expresses the multiplicity of an atom at a prescribed point \(t\) in terms of the moments written in the shifted basis associated with \(x-t\).

\begin{theorem}
  \label{th:theorem-subsequent-moments}
  Let $n,p\in\mathbb{N}$, and let
  \[
    \mathcal{S}
    \coloneqq
    (S_0,S_1,\ldots,S_{2n})
    \in\bigl(\Sym_p(\mathbb{R})\bigr)^{2n+1}
  \]
  be a sequence with moment matrix $M(n)$. Suppose that $\mathcal{S}$
  admits a $(\rank M(n))$--atomic $\mathbb{R}$--representing measure
  $\mu$. Set
  \[
    S_{2n+1}
    \coloneqq
    \int_{\mathbb{R}}x^{2n+1}\,d\mu,
  \]
  and fix $t\in\mathbb{R}$.
For $i\in[0;2n+1]$, define
\[
  \mathcal{E}_t(S_i)
  \coloneqq
  L_{\mathcal{S}}\bigl((x-t)^i\bigr)
  =
  \sum_{\ell=0}^i
  \binom{i}{\ell}(-1)^\ell t^\ell S_{i-\ell},
\]
where, for $i=2n+1$, the operator $L_{\mathcal{S}}$ is extended {to $\RR[x]_{\leq 2n+1}$} by {setting}
$L_{\mathcal{S}}(x^{2n+1})\coloneqq S_{2n+1}$. Set
\[
  \mathcal{T}_t
  \equiv
  \mathcal{T}_t^{(2n)}
  \coloneqq
  \bigl(
    \mathcal{E}_t(S_0),
    \mathcal{E}_t(S_1),
    \ldots,
    \mathcal{E}_t(S_{2n})
  \bigr).
\]
Set
  \[
    \mathcal{H}_t
    \coloneqq
    \bigl(\mathcal{E}_t(S_{i+j-1})\bigr)_{i,j=1}^n,
    \qquad
    K_t
    \coloneqq
    \col\bigl(\mathcal{E}_t(S_{n+i})\bigr)_{i\in[n]}.
  \]
  Then
  \begin{equation}
    \label{th:theorem-subsequent-moments-pt2-condition}
    \mult_{\mu}t
    =
    p-\Card\Dep\!\left(X_{\mathcal{T}_t}^n\right)
    +\Card A_{\mathcal{T}_t}
    -
    \left[
      \rank
      \begin{pmatrix}
        \mathcal{H}_t & K_t\\
        K_t^{\mathsf T} & \mathcal{E}_t(S_{2n+1})
      \end{pmatrix}
      -
      \rank
      \begin{pmatrix}
        \mathcal{H}_t\\
        K_t^{\mathsf T}
      \end{pmatrix}
    \right],
  \end{equation}
  where 
  $
  A_{\mathcal{T}_t}  
  \coloneqq
  \bigcup_{i=1}^n
  \Depn_1\!\left(X_{\mathcal{T}_t}^i\right)
  $
  is as in \eqref{def:def-of-A_t}.
\end{theorem}

\begin{proof}
  Let $m$ denote the right-hand side of
  \eqref{th:theorem-subsequent-moments-pt2-condition}. 
  We need to prove that
  \begin{equation}\label{eq:m-equals-multiplicity}
    m=\mult_\mu t.
  \end{equation}
  
The proof has two parts. First, we construct a matrix-polynomial column
relation whose determinant yields the lower bound
$\mult_{\mu} t\geq m$. Second, we remove the mass at $t$ and compare the
ranks of the resulting flat moment matrices to obtain the reverse
inequality.

  Define
  \[
    \mathcal{H}_{0,t}
    \coloneqq
    \col\bigl(\mathcal{E}_t(S_{i-1})\bigr)_{i\in[n]}.
  \]
  Since $\mu$ is a $(\rank M(n))$--atomic representing measure, its
  moment sequence has a positive semidefinite flat extension $M(n+1)$. More precisely, $M(n+1) \succeq 0$ and $\Rank M(n+1) = \Rank M(n)$. Set
  \[
    S_{2n+2}
    \coloneqq
    \int_{\mathbb{R}}x^{2n+2}\,d\mu
    \qquad\text{and}\qquad
    \mathcal{E}_t(S_{2n+2})
    \coloneqq
    \int_{\mathbb{R}}(x-t)^{2n+2}\,d\mu.
  \]
Write
  \[
    \mathcal{T}_t^{(2n+2)}
    \coloneqq
    \bigl(
      \mathcal{E}_t(S_0),
      \mathcal{E}_t(S_1),
      \ldots,
      \mathcal{E}_t(S_{2n+2})
    \bigr).
  \]
  The moment matrices
  \[
    M(n+1)
    \coloneqq
    \bigl(S_{i+j-2}\bigr)_{i,j=1}^{n+2}
  \qquad
  \text{and}
  \qquad
    M_{\mathcal{T}_t^{(2n+2)}}(n+1)
    \coloneqq
    \bigl(\mathcal{E}_t(S_{i+j-2})\bigr)_{i,j=1}^{n+2}
  \]
  are positive semidefinite and satisfy
  \[
    \rank M_{\mathcal{T}_t^{(2n+2)}}(n+1)
    =
    \rank M_{\mathcal{T}_t}(n)
    =
    \rank M(n).
  \]
  Consequently,
  \[
    \mathcal{C}\!\left(
      \col\bigl(\mathcal{E}_t(S_{n+i})\bigr)_{i\in[n+1]}
    \right)
    \subseteq
    \mathcal{C}\!\left(M_{\mathcal{T}_t}(n)\right).
  \]
  Equivalently,
  \begin{equation}\label{eq:rank-preserving-extension}
    \rank
    \begin{pmatrix}
      \mathcal{H}_{0,t} & \mathcal{H}_t & K_t\\
      \mathcal{E}_t(S_n) & K_t^{\mathsf T}
        & \mathcal{E}_t(S_{2n+1})
    \end{pmatrix}
    =
    \rank
    \begin{pmatrix}
      \mathcal{H}_{0,t} & \mathcal{H}_t\\
      \mathcal{E}_t(S_n) & K_t^{\mathsf T}
    \end{pmatrix}.
  \end{equation}

  Let $P\in M_p(\mathbb{R})$ be a permutation matrix that arranges the
  columns in the order $C_1,C_2,C_3$, where
  \begin{align*}
    C_1
    &\coloneqq
    \Depn_0\!\left(
      X_{\mathcal{T}_t^{(2n+2)}}^{n+1}
    \right),\\
    C_2
    &\coloneqq
    \Depn_1\!\left(
      X_{\mathcal{T}_t^{(2n+2)}}^{n+1}
    \right),\\
    C_3
    &\coloneqq
    \Dep\!\left(
      X_{\mathcal{T}_t^{(2n+2)}}^n
    \right).
  \end{align*}
  Set
  \begin{equation}\label{eq:def-p0-rt-st-q}
    p_0
    \coloneqq
    p-\Card\Dep\!\left(X_{\mathcal{T}_t}^n\right),
    \qquad
    r_t
    \coloneqq
    p_0-m+\Card A_{\mathcal{T}_t},
  \end{equation}
  and
  \[
    s_t
    \coloneqq
    m-\Card A_{\mathcal{T}_t},
    \qquad
    q\coloneqq p-p_0.
  \]
By the definitions of $C_1$, $C_2$, and $C_3$, and by the definitions
of $m$, $r_t$, $s_t$, and $q$, we have
\[
  \Card C_1=r_t,
  \qquad
  \Card C_2=s_t,
  \qquad
  \Card C_3=q.
\]

  Decompose
  \begin{equation}\label{eq:def-widehat-Kt}
    \widehat{K}_t\equiv
    \begin{pmatrix}
      \widehat{K}_{t,1}
      &
      \widehat{K}_{t,2}
      &
      \widehat{K}_{t,3}
    \end{pmatrix}
    \coloneqq
    K_tP,
  \end{equation}
  where the three blocks have $r_t$, $s_t$, and $q$ columns,
  respectively. Also write
  \begin{equation}\label{eq:def-widehat-Zt}
    \widehat{Z}_t
    \coloneqq
    P^{\mathsf T}\mathcal{E}_t(S_{2n+1})P
    \equiv
    \begin{pmatrix}
      \widehat{Z}_{t,1}
        & \widehat{Z}_{t,2}
        & \widehat{Z}_{t,3}\\
      \widehat{Z}_{t,2}^{\mathsf T}
        & \widehat{Z}_{t,4}
        & \widehat{Z}_{t,5}\\
      \widehat{Z}_{t,3}^{\mathsf T}
        & \widehat{Z}_{t,5}^{\mathsf T}
        & \widehat{Z}_{t,6}
    \end{pmatrix}.
  \end{equation}
  The block sizes are
  \begin{align*}
    \widehat{Z}_{t,1}&\in M_{r_t}(\mathbb{R}),&
    \widehat{Z}_{t,2}&\in M_{r_t\times s_t}(\mathbb{R}),&
    \widehat{Z}_{t,3}&\in M_{r_t\times q}(\mathbb{R}),\\
    \widehat{Z}_{t,4}&\in M_{s_t}(\mathbb{R}),&
    \widehat{Z}_{t,5}&\in M_{s_t\times q}(\mathbb{R}),&
    \widehat{Z}_{t,6}&\in M_q(\mathbb{R}).
  \end{align*}

  Since $P$ is a permutation matrix,
  \begin{align}
    \begin{pmatrix}
      \mathcal{H}_t & \widehat{K}_t\\
      \widehat{K}_t^{\mathsf T} & \widehat{Z}_t
    \end{pmatrix}
    &=
    (I_{np}\oplus P^{\mathsf T})
    \begin{pmatrix}
      \mathcal{H}_t & K_t\\
      K_t^{\mathsf T} & \mathcal{E}_t(S_{2n+1})
    \end{pmatrix}
    (I_{np}\oplus P),\label{eq:permuted-HKZ}
  \end{align}
  It follows from the definition of $r_t$ and rank invariance under multiplication by invertible matrices
  that
  \begin{equation}\label{eq:rank-equalities-t}
    \rank
    \begin{pmatrix}
      \mathcal{H}_t & \widehat{K}_{t,1}\\
      \widehat{K}_{t,1}^{\mathsf T} & \widehat{Z}_{t,1}
    \end{pmatrix}
    =
    \rank
    \begin{pmatrix}
      \mathcal{H}_t & \widehat{K}_t\\
      \widehat{K}_t^{\mathsf T} & \widehat{Z}_t
    \end{pmatrix}
    =
    \rank
    \begin{pmatrix}
      \mathcal{H}_t\\
      \widehat{K}_t^{\mathsf T}
    \end{pmatrix}
    +r_t.
  \end{equation}

  Define
  \begin{equation}\label{eq:def-widehat-Tn}
    \widehat{T}_n\equiv
    \begin{pmatrix}
      \widehat{T}_{n,1}
      &
      \widehat{T}_{n,2}
      &
      \widehat{T}_{n,3}
    \end{pmatrix}
    \coloneqq
    \mathcal{E}_t(S_n)P,
  \end{equation}
  where the three blocks have $r_t$, $s_t$, and $q$ columns,
  respectively. Then
  \begin{align}
    \begin{pmatrix}
      \mathcal{H}_{0,t} & \mathcal{H}_t & \widehat{K}_t\\
      \widehat{T}_n^{\mathsf T}
        & \widehat{K}_t^{\mathsf T}
        & \widehat{Z}_t
    \end{pmatrix}
    &=
    (I_{np}\oplus P^{\mathsf T})
    \begin{pmatrix}
      \mathcal{H}_{0,t} & \mathcal{H}_t & K_t\\
      \mathcal{E}_t(S_n) & K_t^{\mathsf T}
        & \mathcal{E}_t(S_{2n+1})
    \end{pmatrix}
    (I_{(n+1)p}\oplus P).
    \label{eq:permuted-extended-matrix}
  \end{align}
  Thus, \eqref{eq:rank-preserving-extension} implies
  \begin{equation}\label{eq:rank-permuted-extension}
    \rank
    \begin{pmatrix}
      \mathcal{H}_{0,t} & \mathcal{H}_t & \widehat{K}_t\\
      \widehat{T}_n^{\mathsf T}
        & \widehat{K}_t^{\mathsf T}
        & \widehat{Z}_t
    \end{pmatrix}
    =
    \rank
    \begin{pmatrix}
      \mathcal{H}_{0,t} & \mathcal{H}_t\\
      \widehat{T}_n^{\mathsf T} & \widehat{K}_t^{\mathsf T}
    \end{pmatrix}
    =
    \rank M_{\mathcal{T}_t}(n).
  \end{equation}

  We next construct a block column relation of
  $M_{\mathcal{T}_t^{(2n+2)}}(n+1)$ whose determinant has a zero at the
  origin of sufficiently high order.

  By Proposition~\ref{pr:t-optimal-block-column-relation-existence},
  there exists a block column relation
  \[
    H(x)=\sum_{i=0}^n x^iH_i
  \]
  of $M_{\mathcal{T}_t}(n)$ satisfying the lower bound
  \eqref{de:optimal-block-column-relation-2}, with
  \begin{equation}\label{eq:structure-of-Hi}
    H_i
    =
    \begin{pmatrix}
      \mathbf{0}_{p\times p_0} & \widehat{H}_i
    \end{pmatrix},
    \qquad i\in[0;n],
  \end{equation}
  and
  \[
    \rank H_n=p-p_0=q.
  \]
  Since $M_{\mathcal{T}_t^{(2n+2)}}(n+1)$ is block--recursively
  generated, the relation $H$ propagates to the extended moment matrix:
  \begin{equation}\label{eq:extended-relation-H}
    \begin{pmatrix}
      \mathcal{H}_t & K_t\\
      K_t^{\mathsf T} & \mathcal{E}_t(S_{2n+1})
    \end{pmatrix}
    \col(H_i)_{i\in[0;n]}
    =
    \mathbf{0}_{(n+1)p\times p}.
  \end{equation}
  Applying the permutation gives
  \begin{equation}\label{eq:permuted-relation-H}
    \begin{pmatrix}
      \mathcal{H}_t & \widehat{K}_t\\
      \widehat{K}_t^{\mathsf T} & \widehat{Z}_t
    \end{pmatrix}
    (I_{np}\oplus P^{\mathsf T})
    \col(H_i)_{i\in[0;n]}
    =
    \mathbf{0}_{(n+1)p\times p}.
  \end{equation}

  In accordance with the partition induced by $P$, write
  \begin{equation}\label{eq:structure-of-Hn}
    P^{\mathsf T}H_n
    =
    \begin{pmatrix}
      \mathbf{0}_{r_t\times p_0} & G_1\\
      \mathbf{0}_{s_t\times p_0} & G_2\\
      \mathbf{0}_{q\times p_0} & G_3
    \end{pmatrix}.
  \end{equation}
  By the definition of $P$, the last $q$ columns correspond precisely
  to the dependent columns of $X_{\mathcal{T}_t}^n$. Since
  $\rank H_n=q$, the matrix $G_3\in M_q(\mathbb{R})$ is invertible.

The first equality in~\eqref{eq:rank-equalities-t} shows that the columns of the block
containing $\widehat{K}_{t,2}$ and the corresponding blocks of
$\widehat{Z}_t$ belong to the column space generated by the preceding
blocks. Hence, there exist matrices
  \[
    J_1\in M_{np\times s_t}(\mathbb{R}),
    \qquad
    J_2\in M_{r_t\times s_t}(\mathbb{R})
  \]
  such that
  \begin{equation}\label{eq:relation-J}
    \begin{pmatrix}
      \mathcal{H}_t
        & \widehat{K}_{t,1}
        & \widehat{K}_{t,2}\\
      \widehat{K}_{t,1}^{\mathsf T}
        & \widehat{Z}_{t,1}
        & \widehat{Z}_{t,2}\\
      \widehat{K}_{t,2}^{\mathsf T}
        & \widehat{Z}_{t,2}^{\mathsf T}
        & \widehat{Z}_{t,4}\\
      \widehat{K}_{t,3}^{\mathsf T}
        & \widehat{Z}_{t,3}^{\mathsf T}
        & \widehat{Z}_{t,5}^{\mathsf T}
    \end{pmatrix}
    \begin{pmatrix}
      -J_1\\
      -J_2\\
      I_{s_t}
    \end{pmatrix}
    =
    \mathbf{0}_{(n+1)p\times s_t}.
  \end{equation}
  Similarly, the first equality in~\eqref{eq:rank-permuted-extension} implies that the relevant block of $r_t$ columns belongs to the column space generated by the preceding blocks. Therefore, there exist matrices
  \[
    U_1\in M_{p\times r_t}(\mathbb{R}),
    \qquad
    U_2\in M_{np\times r_t}(\mathbb{R})
  \]
  such that
  \begin{equation}\label{eq:relation-U}
    \begin{pmatrix}
      \mathcal{H}_{0,t}
        & \mathcal{H}_t
        & \widehat{K}_{t,1}\\
      \widehat{T}_{n,1}^{\mathsf T}
        & \widehat{K}_{t,1}^{\mathsf T}
        & \widehat{Z}_{t,1}\\
      \widehat{T}_{n,2}^{\mathsf T}
        & \widehat{K}_{t,2}^{\mathsf T}
        & \widehat{Z}_{t,2}^{\mathsf T}\\
      \widehat{T}_{n,3}^{\mathsf T}
        & \widehat{K}_{t,3}^{\mathsf T}
        & \widehat{Z}_{t,3}^{\mathsf T}
    \end{pmatrix}
    \begin{pmatrix}
      -U_1\\
      -U_2\\
      I_{r_t}
    \end{pmatrix}
    =
    \mathbf{0}_{(n+1)p\times r_t}.
  \end{equation}
  The second equality in~\eqref{eq:rank-equalities-t} implies
  \begin{equation}\label{eq:rank-U1}
    \rank U_1=r_t.
  \end{equation}

  Combining \eqref{eq:permuted-relation-H},
  \eqref{eq:structure-of-Hi}, \eqref{eq:structure-of-Hn},
  \eqref{eq:relation-J}, and \eqref{eq:relation-U}, we obtain
  \[
    \begin{pmatrix}
      \mathcal{H}_{0,t} & \mathcal{H}_t & \widehat{K}_t\\
      \widehat{T}_n^{\mathsf T}
        & \widehat{K}_t^{\mathsf T}
        & \widehat{Z}_t
    \end{pmatrix}
    \begin{pmatrix}
      -U_1 & \mathbf{0}_{p\times s_t} & \mathbf{0}_{p\times q}\\
      -U_2 & -J_1 & \col(\widehat{H}_i)_{i\in[0;n-1]}\\
      I_{r_t} & -J_2 & G_1\\
      \mathbf{0}_{s_t\times r_t} & I_{s_t} & G_2\\
      \mathbf{0}_{q\times r_t} & \mathbf{0}_{q\times s_t} & G_3
    \end{pmatrix}
    =
    \mathbf{0}_{(n+1)p\times p}.
  \]
  Therefore,
  \begin{equation}\label{eq:relation-widehat-Qn}
    \begin{pmatrix}
      \widehat{K}_t\\
      \widehat{Z}_t
    \end{pmatrix}
    \widehat{Q}_n
    =
    \row\bigl(\widehat{X}_{\mathcal{T}_t}^i\bigr)_{i\in[n]}
    \begin{pmatrix}
      U_2 & J_1 & \col(-\widehat{H}_i)_{i\in[0;n-1]}
    \end{pmatrix}
    +
    \widehat{X}_{\mathcal{T}_t}^0
    \begin{pmatrix}
      U_1 & \mathbf{0}_{p\times(s_t+q)}
    \end{pmatrix},
  \end{equation}
  where
  \begin{equation}\label{eq:def-widehat-Qn}
    \widehat{Q}_n
    \coloneqq
    \begin{pmatrix}
      I_{r_t} & -J_2 & G_1\\
      \mathbf{0}_{s_t\times r_t} & I_{s_t} & G_2\\
      \mathbf{0}_{q\times r_t} & \mathbf{0}_{q\times s_t} & G_3
    \end{pmatrix},
  \end{equation}
  and
  \begin{equation}\label{eq:def-permuted-block-columns}
    \row\bigl(\widehat{X}_{\mathcal{T}_t}^i\bigr)_{i\in[0;n]}
    \coloneqq
    (I_{np}\oplus P^{\mathsf T})
    \row\!\left(
      \col\bigl(\mathcal{E}_t(S_{i+j})\bigr)_{j\in[0;n]}
    \right)_{i\in[0;n]}.
  \end{equation}
  Since $G_3$ is invertible, so is $\widehat{Q}_n$.

  Decompose
  \begin{equation}\label{eq:block-decomposition-U2-J1}
    \begin{pmatrix}
      U_2 & J_1
    \end{pmatrix}
    =
    \begin{pmatrix}
      U_{2,0} & J_{1,0}\\
      U_{2,1} & J_{1,1}\\
      \vdots & \vdots\\
      U_{2,n-1} & J_{1,n-1}
    \end{pmatrix},
  \end{equation}
  where
  \[
    U_{2,i}\in M_{p\times r_t}(\mathbb{R}),
    \qquad
    J_{1,i}\in M_{p\times s_t}(\mathbb{R}).
  \]
  For $i\in[0;n-1]$, define
  \begin{equation}\label{eq:def-Qi}
    Q_i
    \coloneqq
    \begin{pmatrix}
      U_{2,i} & J_{1,i} & -\widehat{H}_i
    \end{pmatrix}.
  \end{equation}
  Then \eqref{eq:relation-widehat-Qn} becomes
  \begin{equation}\label{eq:relation-Qi-permuted}
    \begin{pmatrix}
      \widehat{K}_t\\
      \widehat{Z}_t
    \end{pmatrix}
    \widehat{Q}_n
    =
    \sum_{i=1}^n
      \widehat{X}_{\mathcal{T}_t}^iQ_{i-1}
    +
    \widehat{X}_{\mathcal{T}_t}^0
    \begin{pmatrix}
      U_1 & \mathbf{0}_{p\times(s_t+q)}
    \end{pmatrix}.
  \end{equation}

  Set
  \begin{equation}\label{eq:def-Qn}
    Q_n\coloneqq P\widehat{Q}_n.
  \end{equation}
  Undoing the permutation in~\eqref{eq:relation-Qi-permuted} yields
  \begin{equation}\label{eq:relation-Qi}
    \col\bigl(\mathcal{E}_t(S_{n+i})\bigr)_{i\in[n+1]}Q_n
    =
    \sum_{i=1}^nX_{\mathcal{T}_t}^iQ_{i-1}
    +
    X_{\mathcal{T}_t}^0
    \begin{pmatrix}
      U_1 & \mathbf{0}_{p\times(s_t+q)}
    \end{pmatrix}.
  \end{equation}
  Define
  \begin{equation}\label{eq:def-Q-polynomial}
    Q(x)
    \coloneqq
    x^{n+1}Q_n
    -
    \sum_{i=1}^n x^iQ_{i-1}
    -
    \begin{pmatrix}
      U_1 & \mathbf{0}_{p\times(s_t+q)}
    \end{pmatrix}.
  \end{equation}
  By \eqref{eq:relation-Qi}, the polynomial $Q(x)$ is a block column
  relation of $M_{\mathcal{T}_t^{(2n+2)}}(n+1)$. Hence,
  Proposition~\ref{cor:block-relations-change-of-basis} shows that
  $Q(x-t)$ is a block column relation of $M(n+1)$.

  By \eqref{eq:rank-U1},
  \begin{equation}\label{eq:rank-constant-coefficient}
    \rank
    \begin{pmatrix}
      U_1 & \mathbf{0}_{p\times(s_t+q)}
    \end{pmatrix}
    =
    r_t.
  \end{equation}
Define
\[
  \widetilde{W}_0
  \coloneqq
  \Ker\!\begin{pmatrix}U_1 & \mathbf{0}_{p\times(s_t+q)}\end{pmatrix},
  \qquad
  \widetilde{W}_i
  \coloneqq
  \widetilde{W}_0\cap\bigcap_{j=0}^{i-1}\Ker Q_j,
  \quad i\in[n].
\]
  Since $Q_i$ ends with the block $-\widehat{H}_i$, for every
  $i\in[0;n-1]$ one has
  \begin{equation}\label{eq:kernel-inclusion}
    v\in\bigcap_{j=0}^i\Ker\widehat{H}_j
    \quad\Longrightarrow\quad
    \begin{pmatrix}
      \mathbf{0}_{r_t+s_t}\\
      v
    \end{pmatrix}
    \in\widetilde{W}_{i+1}.
  \end{equation}
  Let
  \[
    \widetilde{s}_i\coloneqq\dim\widetilde{W}_i,
    \qquad i\in[0;n].
  \]
  It follows from~\eqref{eq:kernel-inclusion} that
  \begin{equation}\label{eq:kernel-dimension-inequality}
    \widetilde{s}_i
    \geq
    \dim\left(
      \bigcap_{j=0}^{i-1}\Ker\widehat{H}_j
    \right),
    \qquad i\in[n].
  \end{equation}

Defining
\[
  \Phi(H)\coloneqq
  \sum_{i=0}^n
  \dim\!\left(\bigcap_{j=0}^i\Ker H_j\right),
\]
we have
\begin{equation}\label{eq:lower-bound-H}
  \Phi(H)
  \geq
  (n+1)p_0
  +(n-z_s)q
  +\sum_{j=1}^{s-1}
    (z_{j+1}-z_j)(q-r_{s-j})
  +\Card A_{\mathcal{T}_t}.
\end{equation}

  Moreover, \eqref{eq:structure-of-Hi} implies that, for every
  $i\in[0;n]$,
  \[
    \bigcap_{j=0}^i\Ker H_j
    =
    \mathbb{R}^{p_0}
    \oplus
    \bigcap_{j=0}^i\Ker\widehat{H}_j.
  \]
  Therefore, for every $w\in[0;n]$,
  \begin{equation}\label{eq:kernel-dimension-H-widehat-H}
    \sum_{i=0}^w
      \dim\left(\bigcap_{j=0}^i\Ker H_j\right)
    =
    (w+1)p_0
    +
    \sum_{i=0}^w
      \dim\left(\bigcap_{j=0}^i\Ker\widehat{H}_j\right).
  \end{equation}
We have
\begin{align*}
  \sum_{i=0}^n\widetilde{s}_i
  &\geq
  \widetilde{s}_0
  +
  \sum_{i=1}^n
  \dim\left(
    \bigcap_{j=0}^{i-1}\Ker\widehat{H}_j
  \right)
  &&\text{by~\eqref{eq:kernel-dimension-inequality}}
  \notag\\
  &=
  \widetilde{s}_0
  +
  \sum_{i=0}^{n-1}
  \dim\left(
    \bigcap_{j=0}^i\Ker\widehat{H}_j
  \right)
  &&\text{by reindexing}
  \notag\\
  &=
  s_t+q
  +
  \sum_{i=0}^{n-1}
  \dim\left(
    \bigcap_{j=0}^i\Ker\widehat{H}_j
  \right)
  &&\text{by~\eqref{eq:rank-constant-coefficient}}
  \notag\\
  &=
  s_t+q
  +
  \sum_{i=0}^{n-1}
  \dim\left(
    \bigcap_{j=0}^i\Ker H_j
  \right)
  -np_0
  &&\text{by~\eqref{eq:kernel-dimension-H-widehat-H} with $w=n-1$}
  \notag\\
  &=
  s_t+q+\Phi(H)-(n+1)p_0
  &&\text{since $\Ker\widehat{H}_n=\{\mathbf{0}\}$}
  \notag\\
  &\geq
  m+(n-z_s+1)q
  +
  \sum_{j=1}^{s-1}
    (z_{j+1}-z_j)(q-r_{s-j})
  &&\text{by~\eqref{eq:lower-bound-H}}
  \notag\\
  &=
  m+(n+1)p
  -
  \left(
    (n+1)p_0
    +
    \sum_{j=1}^s z_{s-j+1}p_j
  \right)
  &&\text{by the definitions of $p_j$, $r_j$, and $q$}
  \notag\\
  &=
  m+(n+1)p-\rank M_{\mathcal{T}_t}(n)
  &&\text{by block--recursive generation of $M_{\mathcal{T}_t}(n)$}
  \notag\\
  &=
  m+(n+1)p-\rank M(n)
  &&\text{by Proposition~\ref{prop:change-of-basis}.}
  \label{eq:sum-kernel-dimensions-final}
\end{align*}
  Since $Q_n$ is invertible, Lemma~\ref{lem:determinant-matrix-polynomial}, together with the estimate above, yields 
  \begin{equation}\label{eq:det-Q-factorization} 
    \det Q(x-t) = (x-t)^{m+(n+1)p-\rank M(n)}g(x),
  \end{equation} 
  where $0\neq g(x)\in\mathbb{R}[x]$.
  Moreover,
  \[
    \deg\big(\det Q(x-t)\big)=(n+1)p,
  \]
  and therefore
  \begin{equation}\label{eq:degree-g}
    \deg g=\rank M(n)-m.
  \end{equation}

  Write
  \[
    \mu=\sum_{j=1}^{\ell}\delta_{x_j}A_j,
  \]
  where the points $x_1,\ldots,x_\ell$ are pairwise distinct,
  $A_j\in\Sym_p^{\succeq0}(\mathbb{R})$, and
  \[
    \sum_{j=1}^{\ell}\rank A_j=\rank M(n).
  \]
By Lemma~\ref{lem:support-from-block-column-relation} and
\eqref{eq:det-Q-factorization}, every atom of $\mu$ different from
$t$ is a zero of $g$, and its multiplicity as an atom is at most its
multiplicity as a zero of $g$. By \eqref{eq:degree-g}
and since the total atomic multiplicity of $\mu$ equals
$\operatorname{rank} M(n)$, it follows that
$
    \operatorname{mult}_{\mu} t \geq m.
$
If $t \notin \operatorname{supp}(\mu)$, then
$\operatorname{mult}_{\mu} t=0$, so the preceding inequality implies
that $m=0$. Hence,
$
    \operatorname{mult}_{\mu} t = m.
$
We may therefore assume that $t \in \operatorname{supp}(\mu)$. Write
$t=x_{j'}$ and set
\[
    m' := \operatorname{rank} A_{j'}
        = \operatorname{mult}_{\mu} t.
\]
Suppose, toward a contradiction, that $m'>m$. Define
\[
  \widetilde{\mu}
  \coloneqq
  \mu-\delta_tA_{j'},
\]
and, for $i\in[0;2n+2]$, set
\[
  \widetilde{S}_i
  \coloneqq
  \int_{\mathbb{R}}x^i\,d\widetilde{\mu}.
\]
Let
\[
  \widetilde{\mathcal{S}}
  \coloneqq
  (\widetilde{S}_0,\widetilde{S}_1,\ldots,\widetilde{S}_{2n+2}).
\]

Since
\[
  M(n)-M_{\widetilde{\mathcal{S}}}(n)
\]
is the moment matrix associated with the measure
$\delta_tA_{j'}$, its rank is at most $m'$. Consequently,
\[
  \rank M(n)-m'
  \leq
  \rank M_{\widetilde{\mathcal{S}}}(n)
  \leq
  \rank M_{\widetilde{\mathcal{S}}}(n+1).
\]
On the other hand, $\widetilde{\mu}$ is an atomic representing measure
for $\widetilde{\mathcal{S}}$, and hence
\[
  \rank M_{\widetilde{\mathcal{S}}}(n+1)
  \leq
  \sum_{\substack{j=1\\j\neq j'}}^\ell\rank A_j
  =
  \rank M(n)-m'.
\]
Therefore,
\begin{equation}\label{eq:rank-reduced-moment-matrices}
  \rank M_{\widetilde{\mathcal{S}}}(n)
  =
  \rank M_{\widetilde{\mathcal{S}}}(n+1)
  =
  \rank M(n)-m'.
\end{equation}

For $i\in[0;2n+2]$, define the shifted moments
\[
  \widetilde{T}_i
  \coloneqq
  \int_{\mathbb{R}}(x-t)^i\,d\widetilde{\mu}.
\]
Since the removed mass is supported at $t$, the zeroth shifted moment
changes by $A_{j'}$, whereas all shifted moments of positive degree
remain unchanged. More precisely,
\[
  \widetilde{T}_0
  =
  \mathcal{E}_t(S_0)-A_{j'},
\]
while
\begin{equation}\label{eq:shifted-moments-after-removing-atom}
  \widetilde{T}_i
  =
  \mathcal{E}_t(S_i),
  \qquad i\in[2n+2].
\end{equation}
Thus, $i=0$ is the only index for which the shifted moments associated
with $\widetilde{\mu}$ and $\mu$ differ.

Let
\[
  \widetilde{X}^i
  \coloneqq
  \col\bigl(\widetilde{T}_{i+j}\bigr)_{j\in[0;n+1]},
  \qquad i\in[0;n+1].
\]
By \eqref{eq:rank-reduced-moment-matrices} and
Proposition~\ref{prop:change-of-basis},
\begin{equation}\label{eq:m-prime-tail-rank}
  m'
  =
  \rank M(n)
  -
  \rank
  \row\bigl(\widetilde{X}^i\bigr)_{i\in[n+1]}.
\end{equation}

Let
\[
  X^i
  \coloneqq
  \col\bigl(\mathcal{E}_t(S_{i+j})\bigr)_{j\in[0;n+1]},
  \qquad i\in[0;n+1].
\]
By \eqref{eq:shifted-moments-after-removing-atom},
\[
  \row\bigl(\widetilde{X}^i\bigr)_{i\in[n+1]}
  =
  \row\bigl(X^i\bigr)_{i\in[n+1]}.
\]
Moreover, the definitions of the dependence indices give
\[
  \rank\row\bigl(X^i\bigr)_{i\in[n]}
  =
  \rank M_{\mathcal{T}_t}(n)
  -p_0-\Card A_{\mathcal{T}_t}.
\]
Passing from the block columns indexed by $[n]$ to those indexed by
$[n+1]$ increases the rank by
\[
  r_t
  =
  p_0-m+\Card A_{\mathcal{T}_t}.
\]
It follows that
\begin{align*}
  \rank\row\bigl(\widetilde{X}^i\bigr)_{i\in[n+1]}
  &=
  \rank\row\bigl(X^i\bigr)_{i\in[n+1]}\\
  &=
  \rank M_{\mathcal{T}_t}(n)
  -p_0-\Card A_{\mathcal{T}_t}
  +r_t\\
  &=
  \rank M(n)-m,
\end{align*}
where the last equality also uses
Proposition~\ref{prop:change-of-basis}. Combining this identity with
\eqref{eq:m-prime-tail-rank} yields
\[
  m'
  =
  \rank M(n)-\bigl(\rank M(n)-m\bigr)
  =
  m,
\]
contradicting $m'>m$. Therefore,
$
  \rank A_{j'}\leq m.
$
Together with the previously established inequality
$\rank A_{j'}\geq m$, this gives
$
  \rank A_{j'}=m.
$
Consequently,
\[
  \mult_\mu t=m,
\]
which completes the proof.
\end{proof}

The following immediate consequence of
Theorem~\ref{th:theorem-subsequent-moments} will be used in the proof
of the main result.

\begin{corollary}
  \label{co:theorem-subsequent-moments-corollary-m=0}
  Let $n,p\in\mathbb{N}$, and let
  \[
    \mathcal{S}
    \coloneqq
    (S_0,S_1,\ldots,S_{2n})
    \in\bigl(\Sym_p(\mathbb{R})\bigr)^{2n+1}
  \]
  be a sequence with moment matrix $M(n)$. Suppose that $\mathcal{S}$
  admits a $(\rank M(n))$--atomic $\mathbb{R}$--representing measure
  $\mu$. Set
  $
   \displaystyle S_{2n+1}
    \coloneqq
    \int_{\mathbb{R}}x^{2n+1}\,d\mu,
  $
  fix $t\in\mathbb{R}$, and adopt the notation
  $\mathcal{T}_t$, $\mathcal{H}_t$, $K_t$, and $A_{\mathcal T_t}$ from
  Theorem~\ref{th:theorem-subsequent-moments}.

  Then the following statements are equivalent:
  \begin{enumerate}
    \item \label{co:theorem-subsequent-moments-corollary-m=0-pt1}
    \[
      \mult_\mu t=\Card A_{\mathcal{T}_t}.
    \]

    \item \label{co:theorem-subsequent-moments-corollary-m=0-pt2}
    \[
      \rank
      \begin{pmatrix}
        \mathcal{H}_t & K_t\\
        K_t^{\mathsf T} & \mathcal{E}_t(S_{2n+1})
      \end{pmatrix}
      =
      \rank
      \begin{pmatrix}
        \mathcal{H}_t\\
        K_t^{\mathsf T}
      \end{pmatrix}
      +
      p-\Card\Dep\!\left(X_{\mathcal{T}_t}^n\right).
    \]
  \end{enumerate}
\end{corollary}

\begin{proof}
  By Theorem~\ref{th:theorem-subsequent-moments},
  \[
    \mult_\mu t-\Card A_{\mathcal{T}_t}
    =
    p-\Card\Dep\!\left(X_{\mathcal{T}_t}^n\right)
    -
    \left[
      \rank
      \begin{pmatrix}
        \mathcal{H}_t & K_t\\
        K_t^{\mathsf T} & \mathcal{E}_t(S_{2n+1})
      \end{pmatrix}
      -
      \rank
      \begin{pmatrix}
        \mathcal{H}_t\\
        K_t^{\mathsf T}
      \end{pmatrix}
    \right].
  \]
  Hence, $\mult_\mu t=\Card A_{\mathcal{T}_t}$ if and only if the
  rank identity in \eqref{co:theorem-subsequent-moments-corollary-m=0-pt2} holds.
\end{proof}

The following consequence of
Theorem~\ref{th:theorem-subsequent-moments} applies to arbitrary
finitely atomic representing measures.

\begin{corollary}
  \label{co:theorem-subsequent-moments-finitely-atomic}
  Let $n,p\in\mathbb{N}$, and let
  \[
    \mathcal{S}
    \coloneqq
    (S_0,S_1,\ldots,S_{2n})
    \in\bigl(\Sym_p(\mathbb{R})\bigr)^{2n+1}
  \]
  be a sequence with moment matrix $M(n)$. Suppose that
  $\mathcal{S}$ admits a finitely atomic
  $\mathbb{R}$--representing measure $\mu$. Fix $t\in\mathbb{R}$. 
  Then
  \[
    \mult_\mu t
    \geq
    \Card A_{\mathcal{T}_t},
  \]
  where $A_{\mathcal{T}_t}$ is defined as in Theorem \ref{th:theorem-subsequent-moments}.
\end{corollary}

\begin{proof}
  Extend the moment sequence by setting
  \[
    S_i\coloneqq\int_{\mathbb{R}}x^i\,d\mu,
    \qquad i\geq 2n+1.
  \]
  Since $\mu$ is finitely atomic, there exists $k\in\mathbb{N}$ such
  that
  \[
    \rank M(n+k)
    =
    \rank M(n+k-1)
    =
    \sum_{x\in\supp\mu}\mult_\mu x.
  \]
  Thus, $\mu$ is a $(\rank M(n+k))$--atomic representing measure for
  the extended sequence $(S_0,\ldots,S_{2n+2k})$.
  Set
  \[
    \cT_t \coloneqq (\mathcal{E}_t(S_0), \mathcal{E}_t(S_1), \ldots, \mathcal{E}_t(S_{2n+2k})), \qquad \text{where } \mathcal{E}_t(S_i) \coloneqq \int_{\mathbb{R}}(x-t)^i\,d\mu,
    \quad i\geq 0.
  \]
    For $r\in[n+k]$, define
  \[
    A_{\mathcal{T}_t}(r)
    \coloneqq
    \bigcup_{i=1}^r
    \Depn_1\!\left(X_{\mathcal{T}_t}^i\right).
  \]
  The flatness condition gives
  $
    p
    =
    \Card\Dep\!\left(X_{\mathcal{T}_t}^{n+k}\right).
  $
  Hence, that theorem yields
  \begin{align*}
    \mult_\mu t
    &=
    p-\Card\Dep\!\left(X_{\mathcal{T}_t}^{n+k-1}\right)
    +\Card A_{\mathcal{T}_t}(n+k-1) \\
    &= \Card\Dep\!\left(X_{\mathcal{T}_t}^{n+k}\right)-\Card\Dep\!\left(X_{\mathcal{T}_t}^{n+k-1}\right)
    +\Card A_{\mathcal{T}_t}(n+k-1) \\
    &= \Card\Depn\!\left(X_{\mathcal{T}_t}^{n+k}\right)
    +\Card A_{\mathcal{T}_t}(n+k-1) \\
    &\geq \Card\Depn_{1}\!\left(X_{\mathcal{T}_t}^{n+k}\right)
    +\Card A_{\mathcal{T}_t}(n+k-1) \\
    &= \Card A_{\mathcal{T}_t}(n+k).
  \end{align*}
  Since
  \[
    A_{\mathcal{T}_t}
    =
    A_{\mathcal{T}_t}(n)
    \subseteq
    A_{\mathcal{T}_t}(n+k),
  \]
  we conclude that
  \[
    \mult_\mu t
    \geq
    \Card A_{\mathcal{T}_t}(n+k)
    \geq
    \Card A_{\mathcal{T}_t},
  \]
  as claimed.
\end{proof}

	\section{Minimal $\RR$-representing Measures with Minimal Multiplicities at Prescribed Atoms}
\label{sec:solution-to-the-K-MRTMP}

Let
$
  L\colon\mathbb{R}[x]_{\leq 2n}\to\Sym_p(\mathbb{R})
$
be a linear operator, and let $t_1,\ldots,t_k\in\mathbb{R}$ be
pairwise distinct. In this section, we characterize when $L$ admits a
$\RR$--representing matrix measure $\mu$ satisfying
$
  \mult_\mu t_j
  =
  \Card A_{\mathcal{T}_{t_j}}
$,
  $j\in[k],$
    where $A_{\mathcal{T}_{t_j}}$ is defined
  in~\eqref{def:def-of-A_t}.
By
Corollary~\ref{co:theorem-subsequent-moments-finitely-atomic}, every
finitely atomic representing measure $\nu$ for $L$ satisfies
$
  \mult_\nu t_j
  \geq
  \Card A_{\mathcal{T}_{t_j}},
$
$j\in[k]$.
Thus, the multiplicity of each prescribed point $t_j$ in $\mu$ is the
smallest possible.
We use this result to prove Theorem \ref{thm:main-1}
for $K=\RR$.
\\

We begin with an auxiliary rank lemma.

\begin{lemma}
  \label{le:lemma-for-increasing-rank-v2}
  Let $m,n\in\mathbb{N}$, let $A\in\Sym_m(\mathbb{R})$,
  $B\in M_{m\times n}(\mathbb{R})$, and $P\in\Sym_n(\mathbb{R})$.
  If $P$ is either positive definite or negative definite, then
  \[
    \rank
    \begin{pmatrix}
      A & B\\
      B^{\mathsf T} & B^{\mathsf T}A^\dagger B+P
    \end{pmatrix}
    =
    \rank
    \begin{pmatrix}
      A & B
    \end{pmatrix}
    +n,
  \]
  where $A^\dagger$ denotes the Moore--Penrose pseudoinverse of $A$.
\end{lemma}

\begin{proof}
  Set $r\coloneqq\rank A$. Choose an orthogonal matrix
  $Q\in M_m(\mathbb{R})$ such that
  \[
    Q^{\mathsf T}AQ
    =
    A_1\oplus\mathbf{0}_{m-r},
  \]
  where $A_1\in\Sym_r(\mathbb{R})$ is invertible, and write
  \[
    Q^{\mathsf T}B
    =
    \col(B_1,B_2),
    \qquad
    B_1\in M_{r\times n}(\mathbb{R}),
    \quad
    B_2\in M_{(m-r)\times n}(\mathbb{R}).
  \]
  Since
  \[
    A^\dagger
    =
    Q
    \begin{pmatrix}
      A_1^{-1} & \mathbf{0}\\
      \mathbf{0} & \mathbf{0}
    \end{pmatrix}
    Q^{\mathsf T},
  \]
  one has
  \[
    B^{\mathsf T}A^\dagger B
    =
    B_1^{\mathsf T}A_1^{-1}B_1.
  \]
  Hence, by a congruence transformation and the Schur complement with
  respect to $A_1$,
  \begin{align}
    \rank
    \begin{pmatrix}
      A & B\\
      B^{\mathsf T} & B^{\mathsf T}A^\dagger B+P
    \end{pmatrix}
    &=
    r+
    \rank
    \begin{pmatrix}
      \mathbf{0}_{m-r} & B_2\\
      B_2^{\mathsf T} & P
    \end{pmatrix}.
    \label{eq:rank-reduction-definite-perturbation}
  \end{align}

  Since $P$ is definite, it is invertible. Taking the Schur complement
  with respect to $P$ gives
  \[
    \rank
    \begin{pmatrix}
      \mathbf{0}_{m-r} & B_2\\
      B_2^{\mathsf T} & P
    \end{pmatrix}
    =
    n+\rank\!\left(B_2P^{-1}B_2^{\mathsf T}\right).
  \]
  Because $P^{-1}$ is definite,
  \[
    \rank\!\left(B_2P^{-1}B_2^{\mathsf T}\right)
    =
    \rank B_2.
  \]
  Therefore,
  \begin{equation}\label{eq:rank-completed-block-matrix}
    \rank
    \begin{pmatrix}
      A & B\\
      B^{\mathsf T} & B^{\mathsf T}A^\dagger B+P
    \end{pmatrix}
    =
    r+\rank B_2+n.
  \end{equation}

  On the other hand, the same orthogonal change of basis yields
  \[
    \rank
    \begin{pmatrix}
      A & B
    \end{pmatrix}
    =
    \rank
    \begin{pmatrix}
      A_1 & \mathbf{0} & B_1\\
      \mathbf{0} & \mathbf{0} & B_2
    \end{pmatrix}
    =
    r+\rank B_2.
  \]
  Combining this identity with
  \eqref{eq:rank-completed-block-matrix} proves the result.
\end{proof}

    The following result shows that the multiplicities identified in Theorem~\ref{th:theorem-subsequent-moments} can be attained simultaneously at any finite collection of prescribed points. More precisely, among all representing matrix measures, one can choose a minimal measure whose multiplicity at each prescribed point $t_j$ is the smallest possible value $\Card A_{\mathcal{T}_{t_j}}$.

\begin{theorem}
  \label{th:mainTheorem_RR}
  Let $n,p\in\mathbb{N}$, and let
  $t_1,\ldots,t_{k_1}\in\mathbb{R}$ be pairwise distinct. Let
  \[
    L\colon\mathbb{R}[x]_{\leq 2n}\to\Sym_p(\mathbb{R})
  \]
  be a linear operator with a $\mathbb{R}$--representing matrix measure.
  Then $L$ admits a $(\rank M(n))$--atomic
  $\mathbb{R}$--representing matrix measure $\mu$ such that
  \[
    \mult_\mu t_j
    =
    \Card A_{\mathcal{T}_{t_j}},
    \qquad j\in[k_1],
  \]
  where $A_{\mathcal{T}_{t_j}}$ is defined
  in~\eqref{def:def-of-A_t}.
\end{theorem}

\begin{proof}
  Set
  $S_i\coloneqq L(x^i)$,
  $i\in[0;2n],$
  and write
  $
    \mathcal{S}\coloneqq(S_0,S_1,\ldots,S_{2n}).
  $
  For each $t\in\mathbb{R}$, let $\mathcal{T}_t$ be the shifted
  sequence defined in~\eqref{def:sequenceT}, and adopt the notation
  introduced in Subsections~\ref{subsec:change-of-basis}
  and~\ref{subsec:column-dependencies}.

  For $t\in\mathbb{R}$, write
  \begin{equation}\label{eq:main-R-H0-H-K}
    X_{\mathcal{T}_t}^0
    =:
    \begin{pmatrix}
      \mathcal{H}_{0,t}\\
      \mathcal{E}_t(S_n)
    \end{pmatrix},
    \qquad
    \row\bigl(X_{\mathcal{T}_t}^i\bigr)_{i\in[n]}
    =:
    \begin{pmatrix}
      \mathcal{H}_t\\
      K_t^{\mathsf T}
    \end{pmatrix}
    \in
        \begin{pmatrix}
        \Sym_{np}(\RR)\\[0.2em]
        M_{p\times np}(\RR)
        \end{pmatrix}.
  \end{equation}
  By Proposition~\ref{pr:lemmaDep}, the set
  $\Dep(X_{\mathcal{T}_t}^n)$ is independent of $t$. Set
  \[
    p_0
    \coloneqq
    p-\Card\Dep\!\left(X_{\mathcal{T}_t}^n\right).
  \]
  Choose a permutation matrix $P\in M_p(\mathbb{R})$ such that, upon
  multiplication from the right, the first $p_0$ columns correspond to
  $[p]\setminus\Dep(X_{\mathcal{T}_t}^n)$ and the remaining $p-p_0$
  columns correspond to $\Dep(X_{\mathcal{T}_t}^n)$. For each
  $t\in\mathbb{R}$, define
  \begin{equation}\label{eq:main-R-Khat}
    \widehat{K}_t\equiv
    \begin{pmatrix}
      \widehat{K}_{t,1} & \widehat{K}_{t,3}
    \end{pmatrix}
    \coloneqq
    K_tP,
  \end{equation}
  where $\widehat{K}_{t,1}$ has $p_0$ columns. 
  By the definition of
  $\Dep(X_{\mathcal{T}_t}^n)$,
  \begin{equation}\label{eq:main-R-rank-H-K}
    \rank
    \begin{pmatrix}
      \mathcal{H}_t & \widehat{K}_{t,1}
    \end{pmatrix}
    =
    \rank
    \begin{pmatrix}
      \mathcal{H}_t & \widehat{K}_t
    \end{pmatrix},
    \qquad t\in\mathbb{R}.
  \end{equation}

  Fix $t\in\mathbb{R}$. By
  Proposition~\ref{pr:t-optimal-block-column-relation-existence},
  there exists a block column relation
  \[
    H(x)=\sum_{i=0}^n x^iH_i
  \]
  of $M_{\mathcal{T}_t}(n)$ such that
  \[
    \rank H_n=p-p_0,
    \qquad
    H_i=
    \begin{pmatrix}
      \mathbf{0}_{p\times p_0} & \widehat{H}_i
    \end{pmatrix},
    \qquad i\in[0;n].
  \]
  In the block decomposition
  \begin{equation}\label{eq:main-R-Hn}
    P^{\mathsf T}H_n
    =
    \begin{pmatrix}
      \mathbf{0}_{p_0\times p_0} & G_1\\
      \mathbf{0}_{(p-p_0)\times p_0} & G_3
    \end{pmatrix},
  \end{equation}
  the matrix $G_3\in M_{p-p_0}(\mathbb{R})$ is invertible. Consequently,
  there exists a matrix $\widehat{G}$ such that
  \begin{equation}\label{eq:main-R-dependent-K}
    \widehat{K}_{t,3}
    =
    \begin{pmatrix}
      \mathcal{H}_t & \widehat{K}_{t,1}
    \end{pmatrix}
    \widehat{G}.
  \end{equation}

  For each $j\in[k_1]$, set
  \begin{equation}\label{eq:main-R-Ej}
    E_j
    \coloneqq
    \sum_{i=1}^{2n+1}
      \binom{2n+1}{i}
      (-1)^i
      \bigl(t_j^i-t^i\bigr)\widehat{S}_{2n+1-i,1}
    -
    \widehat{K}_{t_j,1}^{\mathsf T}
    \mathcal{H}_{t_j}^\dagger
    \widehat{K}_{t_j,1},
  \end{equation}
  where
  \[
    \widehat{S}_i\coloneqq P^{\mathsf T}S_iP
  \]
  and $\widehat{S}_{i,1}$ denotes the upper-left
  $p_0\times p_0$ principal submatrix of $\widehat{S}_i$.

  Choose a symmetric matrix
  $\widehat{Z}_{t,1}\in\Sym_{p_0}(\mathbb{R})$ such that
  \begin{equation}\label{eq:main-R-choice-Z1}
    \widehat{Z}_{t,1}+E_j\succ0,
    \qquad j\in[k_1].
  \end{equation}
  Define
  \begin{equation}\label{eq:main-R-Z3-Z6}
    \widehat{Z}_{t,3}
    \coloneqq
    \begin{pmatrix}
      \widehat{K}_{t,1}^{\mathsf T} & \widehat{Z}_{t,1}
    \end{pmatrix}
    \widehat{G},
    \qquad
    \widehat{Z}_{t,6}
    \coloneqq
    \begin{pmatrix}
      \widehat{K}_{t,3}^{\mathsf T}
      &
      \widehat{Z}_{t,3}^{\mathsf T}
    \end{pmatrix}
    \widehat{G},
  \end{equation}
  and set
  \begin{equation}\label{eq:main-R-Zhat}
    \widehat{Z}_t
    \coloneqq
    \begin{pmatrix}
      \widehat{Z}_{t,1} & \widehat{Z}_{t,3}\\
      \widehat{Z}_{t,3}^{\mathsf T} & \widehat{Z}_{t,6}
    \end{pmatrix},
    \qquad
    Z_t\coloneqq P\widehat{Z}_tP^{\mathsf T}.
  \end{equation}
  The matrix $\widehat{Z}_t$ is symmetric. Indeed,
  \eqref{eq:main-R-dependent-K} and
  \eqref{eq:main-R-Z3-Z6} give
  \[
    \widehat{Z}_{t,6}
    =
    \widehat{G}^{\mathsf T}
    \begin{pmatrix}
      \mathcal{H}_t & \widehat{K}_{t,1}\\
      \widehat{K}_{t,1}^{\mathsf T} & \widehat{Z}_{t,1}
    \end{pmatrix}
    \widehat{G}.
  \] 

We next verify that the choice of $\widehat Z_{t,1}$ does not
obstruct the construction of a flat extension. More precisely, we
claim that
\begin{equation}
\label{eq:range-Zt}
\operatorname{rank}
\begin{pmatrix}
\mathcal H_{0,t} & \mathcal H_t & \widehat K_t\\
 \mathcal E_t(S_n) & \widehat K_t^{\mathsf T} & \widehat Z_t
\end{pmatrix}
=
\operatorname{rank}M_{\mathcal T_t}(n).
\end{equation}
In particular,
\begin{equation}
\label{inclusion-for-psd}
\mathcal C\bigl(\operatorname{col}(K_t,Z_t)\bigr)
   \subseteq \mathcal C\bigl(M_{\mathcal T_t}(n)\bigr).
\end{equation}

Indeed, by \eqref{eq:main-R-dependent-K}, \eqref{eq:main-R-Z3-Z6}, and \eqref{eq:main-R-Zhat}, we have
\[
\begin{pmatrix}
 \widehat K_{t,3}\\
 \widehat Z_{t,3}
\end{pmatrix}
=
\begin{pmatrix}
\mathcal H_t & \widehat K_{t,1}\\
 \widehat K_{t,1}^{\mathsf T} & \widehat Z_{t,1}
\end{pmatrix}\widehat G.
\qquad \text{and} \qquad
\widehat Z_{t,6}
=
\begin{pmatrix}
 \widehat K_{t,3}^{\mathsf T}
 &
 \widehat Z_{t,3}^{\mathsf T}
\end{pmatrix}\widehat G.
\]
Consequently, the blocks $\widehat K_{t,3}$,
$\widehat Z_{t,3}$, and $\widehat Z_{t,6}$ do not increase
the rank, and hence
\begin{align}
\operatorname{rank}
\begin{pmatrix}
 \mathcal H_{0,t} & \mathcal H_t & \widehat K_t\\
 \mathcal E_t(S_n) & \widehat K_t^{\mathsf T} & \widehat Z_t
\end{pmatrix}
&=
\operatorname{rank}
\begin{pmatrix}
 \mathcal H_{0,t} & \mathcal H_t & \widehat K_{t,1}\\
 \mathcal E_t(S_n) & \widehat K_{t,1}^{\mathsf T}
                         & \widehat Z_{t,1}
\end{pmatrix}.
\label{eq:rank-reduction-Zt}
\end{align}

We prove that the rank on the right-hand side of
\eqref{eq:rank-reduction-Zt} equals
$\operatorname{rank}M_{\mathcal T_t}(n)$.
Set
\[
\widehat T_n\equiv
\begin{pmatrix}
\widehat T_{n,1} & \widehat T_{n,3}
\end{pmatrix}
\coloneqq
\mathcal E_t(S_n)P,
\]
where $\widehat T_{n,1}$ has $p_0$ columns. By the same column
relations that give
\[
\widehat K_{t,3}
=
\begin{pmatrix}
\mathcal H_t & \widehat K_{t,1}
\end{pmatrix}\widehat G,
\]
we also have
\[
\widehat T_{n,3}
=
\begin{pmatrix}
\mathcal H_{0,t}^{\mathsf T} & \widehat T_{n,1}
\end{pmatrix}\widehat G.
\]
Consequently,
\begin{equation}
\label{eq:rank-MT-selected}
\operatorname{rank}M_{\mathcal T_t}(n)
=
\operatorname{rank}
\begin{pmatrix}
\mathcal H_{0,t} & \mathcal H_t\\
\widehat T_{n,1}^{\mathsf T}
  & \widehat K_{t,1}^{\mathsf T}
\end{pmatrix}.
\end{equation}
Moreover, by the definition of $p_0$,
\begin{equation}
\label{eq:rank-MT-p0}
\operatorname{rank}M_{\mathcal T_t}(n)
=
\operatorname{rank}
\begin{pmatrix}
\mathcal H_{0,t} & \mathcal H_t
\end{pmatrix}
+p_0.
\end{equation}

{Permute the rows of $\begin{pmatrix}
\mathcal H_{0,t} & \mathcal H_t
\end{pmatrix}$ using a permutation matrix $P_2\in M_{np}(\mathbb R)$ such that
\[
P_2^T\begin{pmatrix}
\mathcal H_{0,t} & \mathcal H_t
\end{pmatrix}
=
\begin{pmatrix}
\mathcal H_{0,t,1} & \mathcal H_{t,1}
\\
\mathcal H_{0,t,2} & \mathcal H_{t,2}
\end{pmatrix},
\]
where
\begin{equation}
\label{eq:rank-upper-selection}
\operatorname{rank}
\begin{pmatrix}
\mathcal H_{0,t,1} & \mathcal H_{t,1}\\
\mathcal H_{0,t,2} & \mathcal H_{t,2}
\end{pmatrix}
=
\operatorname{rank}
\begin{pmatrix}
\mathcal H_{0,t,2} & \mathcal H_{t,2}
\end{pmatrix}.
\end{equation}
}
Apply the same permutation to $\widehat K_{t,1}^{\mathsf T}$ and
write
\[
\widehat K_{t,1}^{\mathsf T}P_2
=:
\begin{pmatrix}
(\widehat K_{t,1}^{(1)})^{\mathsf T}
&
(\widehat K_{t,1}^{(2)})^{\mathsf T}
\end{pmatrix}.
\]
By \eqref{eq:rank-MT-selected} and \eqref{eq:rank-MT-p0}, the
permutation $P_2$ can be chosen so that
\begin{equation}
\label{eq:kernel-selected}
\Ker
\begin{pmatrix}
\mathcal H_{0,t,2}^{\mathsf T} & \widehat T_{n,1}\\
\mathcal H_{t,2}^{\mathsf T} & \widehat K_{t,1}
\end{pmatrix}
=
\{0\}.
\end{equation}
Equivalently, the matrix
\[
\begin{pmatrix}
\mathcal H_{0,t,2} & \mathcal H_{t,2}\\
\widehat T_{n,1}^{\mathsf T}
  & \widehat K_{t,1}^{\mathsf T}
\end{pmatrix}
\]
is surjective. Therefore, for every
$\widehat Z_{t,1}\in \Sym_{p_0}(\mathbb R)$,
\begin{align}
\label{eq:arbitrary-Zt1}
&\operatorname{rank}
\begin{pmatrix}
\mathcal H_{0,t,2} & \mathcal H_{t,2} & \widehat K_{t,1}^{(2)}\\
\widehat T_{n,1}^{\mathsf T}
  & \widehat K_{t,1}^{\mathsf T}
  & \widehat Z_{t,1}
\end{pmatrix}
=
\operatorname{rank}
\begin{pmatrix}
\mathcal H_{0,t,2} & \mathcal H_{t,2}\\
\widehat T_{n,1}^{\mathsf T}
  & \widehat K_{t,1}^{\mathsf T}
\end{pmatrix}.
\end{align}

Since the original sequence admits an $\mathbb R$--representing
measure, it admits a positive semidefinite extension. Thus, there
exist matrices $Z_t^{(e)}$ and $W_t^{(e)}$ such that
\[
\begin{pmatrix}
M_{\mathcal T_t}(n)
  & \operatorname{col}(K_t,Z_t^{(e)})\\
\operatorname{row}\bigl(
K_t^{\mathsf T},(Z_t^{(e)})^{\mathsf T}
\bigr)
  & W_t^{(e)}
\end{pmatrix}
\succeq0.
\]
The generalized Schur complement criterion \cite{Alb69} implies
\[
\mathcal C\bigl(\operatorname{col}(K_t,Z_t^{(e)})\bigr)
\subseteq
\mathcal C\bigl(M_{\mathcal T_t}(n)\bigr).
\]
In particular,
\begin{equation}
\label{eq:upper-K-range}
\operatorname{rank}
\begin{pmatrix}
\mathcal H_{0,t,1} & \mathcal H_{t,1} & \widehat K_{t,1}^{(1)}\\
\mathcal H_{0,t,2} & \mathcal H_{t,2} & \widehat K_{t,1}^{(2)}
\end{pmatrix}
=
\operatorname{rank}
\begin{pmatrix}
\mathcal H_{0,t,1} & \mathcal H_{t,1}\\
\mathcal H_{0,t,2} & \mathcal H_{t,2}
\end{pmatrix}.
\end{equation}
By Lemma~\ref{le:rank-completion}, 
\eqref{eq:rank-upper-selection},
\eqref{eq:arbitrary-Zt1}, and
\eqref{eq:upper-K-range}, we obtain
\begin{align*}
&\operatorname{rank}
\begin{pmatrix}
\mathcal  H_{0,t} & \mathcal H_t & \widehat K_{t,1}\\
\widehat T_{n,1}^{\mathsf T}
  & \widehat K_{t,1}^{\mathsf T}
  & \widehat Z_{t,1}
\end{pmatrix}
=
\operatorname{rank}
\begin{pmatrix}
\mathcal H_{0,t} & \mathcal H_t\\
\widehat T_{n,1}^{\mathsf T}
  & \widehat K_{t,1}^{\mathsf T}
\end{pmatrix}
=
\operatorname{rank}M_{\mathcal T_t}(n),
\end{align*}
where the last equality follows from
\eqref{eq:rank-MT-selected}. Together with
\eqref{eq:rank-reduction-Zt}, this proves
\eqref{eq:range-Zt}.


  For each $j\in[k_1]$, define
  \begin{equation}\label{eq:main-R-Ztj}
    \widehat{Z}_{t_j}
    \coloneqq
    \widehat{Z}_t
    +
    \sum_{i=1}^{2n+1}
      \binom{2n+1}{i}
      (-1)^i
      \bigl(t_j^i-t^i\bigr)\widehat{S}_{2n+1-i}.
  \end{equation}
  Let $\widehat{Z}_{t_j,1}$ denote its upper-left
  $p_0\times p_0$ principal submatrix. By
  \eqref{eq:main-R-Ej} and \eqref{eq:main-R-Ztj},
  \[
    \widehat{Z}_{t_j,1}
    =
    \widehat{K}_{t_j,1}^{\mathsf T}
    \mathcal{H}_{t_j}^\dagger
    \widehat{K}_{t_j,1}
    +
    \widehat{Z}_{t,1}+E_j.
  \]
  Hence, by \eqref{eq:main-R-choice-Z1} and
  Lemma~\ref{le:lemma-for-increasing-rank-v2}
  (applied with
$A=\cH_{t_j}$, $B=\widehat K_{t_j,1}$, and $P=\widehat{Z}_{t,1}+E_j$),
  \begin{equation}\label{eq:main-R-rank-small}
    \rank
    \begin{pmatrix}
      \mathcal{H}_{t_j} & \widehat{K}_{t_j,1}\\
      \widehat{K}_{t_j,1}^{\mathsf T}
        & \widehat{Z}_{t_j,1}
    \end{pmatrix}
    =
    \rank
    \begin{pmatrix}
      \mathcal{H}_{t_j}\\
      \widehat{K}_{t_j,1}^{\mathsf T}
    \end{pmatrix}
    +p_0,
    \qquad j\in[k_1].
  \end{equation}

  We now extend the shifted moment sequence by setting
  \begin{equation}\label{eq:main-R-extended-shifted-moments}
    T_{2n+1}\coloneqq Z_t,
    \qquad
    T_{2n+2}
    \coloneqq
    \row(K_t^{\mathsf T},Z_t)
    M_{\mathcal{T}_t}(n)^\dagger
    \col(K_t,Z_t).
  \end{equation}
  Let
  \[
    \mathcal{T}_t^{(2n+2)}
    \coloneqq
    (T_0,T_1,\ldots,T_{2n+2}).
  \]
  By \cite{Alb69}, \eqref{inclusion-for-psd} and 
  \eqref{eq:main-R-extended-shifted-moments}, the corresponding moment matrix
  $M_{\mathcal{T}_t^{(2n+2)}}(n+1)$ is positive semidefinite and
  \[
    \rank M_{\mathcal{T}_t^{(2n+2)}}(n+1)
    =
    \rank M_{\mathcal{T}_t}(n).
  \]
  Define the unshifted moments $S_{2n+1}$ and $S_{2n+2}$
        \begin{align}
        \label{def-S2n+1-S2n+2-main-proof}
        \begin{split}
        S_{2n+1}
        &\coloneqq \cE_t^{-1}(T_{2n+1})
        = T_{2n+1}-\sum_{i=1}^{2n+1}\binom{2n+1}{i}(-1)^i t^i S_{2n+1-i}, \\
        S_{2n+2}
        &\coloneqq \cE_t^{-1}(T_{2n+2})
        = T_{2n+2}-\sum_{i=1}^{2n+2}\binom{2n+2}{i}(-1)^i t^i S_{2n+2-i},
        \end{split}
        \end{align}
        and set
        \begin{equation}
            \label{extended-sequence-cS2n+2}
            \cS^{(2n+2)} \coloneqq (S_0, S_1, \ldots, S_{2n+2}).
        \end{equation} 
  Proposition~\ref{prop:change-of-basis} yields
  \[
    M_{\mathcal{S}^{(2n+2)}}(n+1)\succeq0,
    \qquad
    \rank M_{\mathcal{S}^{(2n+2)}}(n+1)
    =
    \rank M(n).
  \]
  Hence $\mathcal{S}^{(2n+2)}$ admits a
  $(\rank M(n))$--atomic $\mathbb{R}$--representing measure
  \[
    \mu=\sum_{\ell=1}^N\delta_{x_\ell}A_\ell.
  \]
  Its restriction to moments of degree at most $2n$ is an
  $\mathbb{R}$--representing measure for $L$.

  It remains to determine the multiplicities of the prescribed points.
  By the inverse change-of-basis formula,
  \begin{equation}
  \label{correspondence-to-prove}
    P\widehat{Z}_{t_j}P^{\mathsf T}
    =
    \mathcal{E}_{t_j}(S_{2n+1}),
    \qquad j\in[k_1].
  \end{equation}

      Indeed, fix $j\in[k_1]$. Using \eqref{eq:main-R-Zhat} and \eqref{eq:main-R-Ztj}, we have
        \begin{align*}
        \widehat Z_{t_j}
        &=
        \widehat Z_t
        +P^T\sum_{i=1}^{2n+1}\binom{2n+1}{i}(-1)^i\bigl(t_j^{\,i}-t^{\,i}\bigr)S_{2n+1-i}P .
        \end{align*}
    By \eqref{eq:main-R-Zhat}, \eqref{eq:main-R-extended-shifted-moments}, and
        \eqref{def-S2n+1-S2n+2-main-proof}, we have
        \[
        \widehat Z_t = P^T T_{2n+1} P = P^T\mathcal{E}_t(S_{2n+1})P
        = P^T\!\left(S_{2n+1}
        +\sum_{i=1}^{2n+1}\binom{2n+1}{i}(-1)^{\,i}t^iS_{2n+1-i}\right)\!P .
        \]
        Substituting this into the previous expression yields
        \begin{align*} 
        \widehat Z_{t_j} 
        &= P^T\!\left( S_{2n+1} +\sum_{i=1}^{2n+1}\binom{2n+1}{i}(-1)^it^iS_{2n+1-i} \right)\!P \\ 
        &\quad\qquad+ P^T\sum_{i=1}^{2n+1} \binom{2n+1}{i}(-1)^i \bigl(t_j^i-t^i\bigr)S_{2n+1-i}P \\ 
        &= P^T\!\left( S_{2n+1} +\sum_{i=1}^{2n+1}\binom{2n+1}{i}(-1)^it_j^iS_{2n+1-i} \right)\!P \\ 
        &= P^T\,\mathcal E_{t_j}(S_{2n+1})\,P . 
        \end{align*}
        Multiplying by $P$ on the left and by $P^T$ on the right give \eqref{correspondence-to-prove}.

Fix $j\in[k_1]$ and define the extended shifted sequence
\[
  \mathcal{T}_{t_j}^{(2n+2)}
  \coloneqq
  \bigl(
    \mathcal{E}_{t_j}(S_0),
    \mathcal{E}_{t_j}(S_1),
    \ldots,
    \mathcal{E}_{t_j}(S_{2n+2})
  \bigr).
\]
By Proposition~\ref{pr:lemmaDep},
\begin{equation}\label{dep-equality-extended}
  \Dep\!\left(X_{\mathcal{T}_t^{(2n+2)}}^n\right)
  =
  \Dep\!\left(X_{\mathcal{T}_{t_j}^{(2n+2)}}^n\right).
\end{equation}
Moreover, 
\[
  X_{\mathcal{T}_{t_j}^{(2n+2)}}^n
  =
  \begin{pmatrix}
    \mathcal{E}_{t_j}(S_n)\\
    K_{t_j}\\
    Z_{t_j}
  \end{pmatrix}.
\]
Since $P$ places the dependent columns in the last $p-p_0$ positions,
\eqref{dep-equality-extended} implies that the last $p-p_0$ columns of
$X_{\mathcal{T}_{t_j}^{(2n+2)}}^nP$ are linear combinations of the
preceding columns. Consequently, the blocks
$\widehat{K}_{t_j,3}$, $\widehat{Z}_{t_j,3}$, and
$\widehat{Z}_{t_j,6}$ do not increase the column rank. Since the full
matrix is symmetric, the corresponding block rows do not increase its
rank either. Therefore,
\begin{equation}\label{eq:main-R-rank-reduction}
  \rank
  \begin{pmatrix}
    \mathcal{H}_{t_j} & \widehat{K}_{t_j}\\
    \widehat{K}_{t_j}^{\mathsf T} & \widehat{Z}_{t_j}
  \end{pmatrix}
  =
  \rank
  \begin{pmatrix}
    \mathcal{H}_{t_j} & \widehat{K}_{t_j,1}\\
    \widehat{K}_{t_j,1}^{\mathsf T}
      & \widehat{Z}_{t_j,1}
  \end{pmatrix}.
\end{equation}

  Moreover, the columns corresponding to
  $\widehat{K}_{t_j,3}$ and the associated blocks of
  $\widehat{Z}_{t_j}$ are linear combinations of the columns
  corresponding to $\mathcal{H}_{t_j}$ and
  $\widehat{K}_{t_j,1}$. Therefore,
\begin{align*}
  &\rank
  \begin{pmatrix}
    \mathcal{H}_{t_j} & K_{t_j}\\
    K_{t_j}^{\mathsf T} & \mathcal{E}_{t_j}(S_{2n+1})
  \end{pmatrix}\\
  =&
  \rank
  \begin{pmatrix}
    \mathcal{H}_{t_j} & \widehat{K}_{t_j}\\
    \widehat{K}_{t_j}^{\mathsf T} & \widehat{Z}_{t_j}
  \end{pmatrix}
  &&\text{(by the permutation congruence induced by $P$)}
  \\
  =&
  \rank
  \begin{pmatrix}
    \mathcal{H}_{t_j} & \widehat{K}_{t_j,1}\\
    \widehat{K}_{t_j,1}^{\mathsf T}
      & \widehat{Z}_{t_j,1}
  \end{pmatrix}
  &&\text{by }\eqref{eq:main-R-rank-reduction}
  \\
  =&
  \rank
  \begin{pmatrix}
    \mathcal{H}_{t_j}\\
    \widehat{K}_{t_j,1}^{\mathsf T}
  \end{pmatrix}
  +p_0
  &&\text{by~\eqref{eq:main-R-rank-small}}
  \\
  =&
  \rank
  \begin{pmatrix}
    \mathcal{H}_{t_j}\\
    \widehat{K}_{t_j}^{\mathsf T}
  \end{pmatrix}
  +p_0
  &&\text{by~\eqref{eq:main-R-rank-H-K}}
  \\
  =&
  \rank
  \begin{pmatrix}
    \mathcal{H}_{t_j}\\
    K_{t_j}^{\mathsf T}
  \end{pmatrix}
  +
  p-\Card\Dep\!\left(X_{\mathcal{T}_{t_j}}^n\right)
  &&\text{by $\widehat{K}_{t_j}=K_{t_j}P$ and the definition of $p_0$.}
\end{align*}

  Corollary~\ref{co:theorem-subsequent-moments-corollary-m=0} now
  implies that
  \[
    \mult_\mu t_j
    =
    \Card A_{\mathcal{T}_{t_j}},
    \qquad j\in[k_1].
  \]
  This completes the proof.
\end{proof}

The following corollary characterizes when the representing measure in
Theorem~\ref{th:mainTheorem_RR} can be chosen so that none of the
prescribed points is an atom.

\begin{corollary}
  \label{co:corollary_mainTheorem_RR}
  Let $n,p\in\mathbb{N}$, and let
  $t_1,\ldots,t_{k_1}\in\mathbb{R}$ be pairwise distinct. Let
  \[
    L\colon\mathbb{R}[x]_{\leq 2n}\to\Sym_p(\mathbb{R})
  \]
  be a linear operator which 
  admits an $\mathbb{R}$--representing matrix measure. Then the following
  statements are equivalent:
  \begin{enumerate}
    \item\label{co:corollary_mainTheorem_RR-pt0}
    There exists a finitely atomic
    $\mathbb{R}$--representing matrix measure $\mu$ for $L$ such that
    \[
      t_j\notin\supp(\mu),
      \qquad j\in[k_1].
    \]
    
    \item\label{co:corollary_mainTheorem_RR-pt1}
    There exists a $(\rank M(n))$--atomic
    $\mathbb{R}$--representing matrix measure $\mu$ for $L$ such that
    \[
      t_j\notin\supp(\mu),
      \qquad j\in[k_1].
    \]

    \item\label{co:corollary_mainTheorem_RR-pt2}
    For every $j\in[k_1]$, we have
    \[
      A_{\mathcal{T}_{t_j}}=\emptyset,
    \]
    where $A_{\mathcal{T}_{t_j}}$ is defined
    in~\eqref{def:def-of-A_t}.
  \end{enumerate}
\end{corollary}

\begin{proof}
  Assume first that
  \eqref{co:corollary_mainTheorem_RR-pt0} holds. Then
  $\mult_\mu t_j=0$,
    $j\in[k_1].$
  It follows from Corollary \ref{co:theorem-subsequent-moments-finitely-atomic} that
  $A_{\mathcal{T}_{t_j}}=\emptyset$ for every $j\in[k_1]$. Hence,
  \eqref{co:corollary_mainTheorem_RR-pt2} holds.

  Assume that
  \eqref{co:corollary_mainTheorem_RR-pt2} holds. By
  Theorem~\ref{th:mainTheorem_RR}, there exists a
  $(\rank M(n))$--atomic $\mathbb{R}$--representing matrix measure
  $\mu$ for $L$ such that
  $
    \mult_\mu t_j
    =
    \Card A_{\mathcal{T}_{t_j}}
    =
    0,
    $
    $j\in[k_1].$
  Therefore, $t_j\notin\supp(\mu)$ for every $j\in[k_1]$, proving
  \eqref{co:corollary_mainTheorem_RR-pt1}.

  The implication
  $\eqref{co:corollary_mainTheorem_RR-pt1}\Rightarrow\eqref{co:corollary_mainTheorem_RR-pt0}$ is clear.
\end{proof}

We are now ready to prove Theorem~\ref{thm:main-1} in the case
$K=\RR$.

\begin{proof}[Proof of Theorem~\ref{thm:main-1} for $K=\RR$]
  The implication
  \eqref{thm:main-1-i} $\Rightarrow$ \eqref{thm:main-1-ii}
  follows from~\cite[Corollary~5.3]{MS23+}.

  We next prove
  \eqref{thm:main-1-ii} $\Rightarrow$ \eqref{thm:main-1-iv}.
  Let $\nu$ be a finitely atomic $\RR$-representing measure for $\mathcal L$. By Proposition 
  \ref{pr:proposition-measure-sets}, the measure $\mu :=q^{-1}\cdot\nu$ is a finitely atomic $\RR$-representing measure for $L$ and satisfies $\mu(\Lambda)=\mathbf 0_p$.
  By Theorem \ref{th:Hamburger-matricial},
  \eqref{thm:main-1-iv-a} and \eqref{thm:main-1-iv-b} hold.
  It remains to prove \eqref{thm:main-1-iv-c}.
  Consider the extended moment
  sequence
  \[
    \widetilde{\mathcal{S}}
    \coloneqq
    (S_0,\ldots,S_{2n+2k}),
    \qquad
    S_i\coloneqq\int_{\RR}x^i\,d\mu,
  \]
  satisfying
  \[
    \rank M_{\widetilde{\mathcal{S}}}(n+k)
    =
    \rank M_{\widetilde{\mathcal{S}}}(n+k-1).
  \]

  Fix $t_j\in\Lambda$. In what follows, let
  $\mathcal{H}_{t_j}$ and $K_{t_j}$ denote the matrices defined as in
  Theorem~\ref{th:theorem-subsequent-moments}, with
  $\widetilde{\mathcal{S}}$ in place of $\mathcal{S}$ and $n+k-1$ in
  place of $n$. The flatness condition implies that
  \[
    p
    =
    \Card\Dep\!\left(
      X_{\mathcal{T}_{t_j}}^{n+k}
    \right)
  \]
  and
  \[
    \rank
    \begin{pmatrix}
      \mathcal{H}_{t_j} & K_{t_j}\\
      K_{t_j}^{\mathsf T}
        & \mathcal{E}_{t_j}(S_{2n+2k-1})
    \end{pmatrix}
    =
    \rank
    \begin{pmatrix}
      \mathcal{H}_{t_j}\\
      K_{t_j}^{\mathsf T}
    \end{pmatrix}.
  \]
  Thus, condition~{\rm(2)} of
  Corollary~\ref{co:theorem-subsequent-moments-corollary-m=0}
  is satisfied. Consequently,
  \[
    \mult_\mu t_j
    =
    \Card A_{\mathcal{T}_{t_j}}(n+k),
  \]
  where
  \[
    A_{\mathcal{T}_{t_j}}(r)
    \coloneqq
    \bigcup_{i=1}^r
    \Depn_1\!\left(X_{\mathcal{T}_{t_j}}^i\right),
    \qquad r\in[n+k].
  \]
  Since $t_j\in\Lambda$, we have $\mult_\mu t_j=0$, and hence
  \[
    A_{\mathcal{T}_{t_j}}(n+k)=\emptyset.
  \]
  The sets $A_{\mathcal{T}_{t_j}}(r)$ are increasing in $r$, so
  \[
    A_{\mathcal{T}_{t_j}}(r)=\emptyset,
    \qquad r\in[n+k].
  \]
  Since $t_j\in\Lambda$ was arbitrary, this proves
  \eqref{thm:main-1-iv}.

  The implication
  \eqref{thm:main-1-iv} $\Rightarrow$ \eqref{thm:main-1-iii}
  follows from
  Corollary~\ref{co:corollary_mainTheorem_RR}.
  Finally, the implication
  \eqref{thm:main-1-iii} $\Rightarrow$ \eqref{thm:main-1-ii}
  is immediate.
\end{proof}

	\section{Minimal $[0,\infty)$-representing Measures with Minimal Multiplicities at Prescribed Atoms}
\label{sec:solution-to-the-K-MRTMP-v2}

In this section, we establish analogues of the results from Section~\ref{sec:solution-to-the-K-MRTMP} for the case where the support of  representing measures is contained in $[0,\infty)$. We then use these results to complete the proof of Theorem~\ref{thm:main-1} for $K=[0,\infty)$.\\

The following result is an analogue of Theorem \ref{th:mainTheorem_RR}, where $\RR$ is replaced by $[0,\infty)$.

\begin{theorem}
  \label{th:mainTheorem_halfline}
  Let $n,p\in\mathbb{N}$, and let
  $t_1,\ldots,t_{k_1}\in[0,\infty)$ be pairwise distinct. Let
  \[
    L\colon\mathbb{R}[x]_{\leq 2n}\to\Sym_p(\mathbb{R})
  \]
  be a linear operator with an $[0,\infty)$--representing matrix
  measure. 
  Then $L$ admits a $(\rank M(n))$--atomic
  $[0,\infty)$--representing matrix measure $\mu$ such that
  \[
    \mult_\mu t_j
    =
    \Card A_{\mathcal{T}_{t_j}},
    \qquad j\in[k_1],
  \]
  where $A_{\mathcal{T}_{t_j}}$ is defined
  in~\eqref{def:def-of-A_t}.
  

\end{theorem}

\begin{proof}
    The construction of the flat extension of $M(n)$ must meet two requirements simultaneously. First, we choose $L(x^{2n+1})$ so that the rank identities at all prescribed points hold. Second, we choose the same matrix sufficiently positive to ensure the Stieltjes localizing condition. The remainder of the argument then completes the flat extension as in the proof of Theorem \ref{th:mainTheorem_RR}.

    Set
  $S_i\coloneqq L(x^i)$,
  $i\in[0;2n],$
  and write
  $
    \mathcal{S}\coloneqq(S_0,S_1,\ldots,S_{2n}).
  $
  For each $t\in\mathbb{R}$, let $\mathcal{T}_t$ be the shifted
  sequence defined in~\eqref{def:sequenceT}, and adopt the notation
  introduced in Subsections~\ref{subsec:change-of-basis}
  and~\ref{subsec:column-dependencies}. We follow the construction
  in the proof of Theorem~\ref{th:mainTheorem_RR}, taking the reference
  point $t=0$, up to the choice of the block
  $\widehat{Z}_{0,1}$. We retain all notation introduced there. In
  particular,
  \[
    \widehat{K}_0
    =
    \begin{pmatrix}
      \widehat{K}_{0,1} & \widehat{K}_{0,3}
    \end{pmatrix},
  \]
  where $\widehat{K}_{0,1}$ has $p_0$ columns, and
  \begin{equation}\label{eq:halfline-K03-dependence}
    \widehat{K}_{0,3}
    =
    \begin{pmatrix}
      \mathcal{H}_0 & \widehat{K}_{0,1}
    \end{pmatrix}
    \widehat{G}.
  \end{equation}
  Since the reference point is $t=0$, we have
  \[
    \mathcal{H}_0=\mathcal{H}_x(n-1).
  \]

  In addition to the rank conditions required in the proof of
  Theorem~\ref{th:mainTheorem_RR}, we must choose the subsequent moment
  $S_{2n+1}$ so that
  \begin{equation}\label{eq:halfline-Schur-condition}
    S_{2n+1}
    -
    \row(S_{n+i})_{i\in[n]}
    \mathcal{H}_x(n-1)^\dagger
    \col(S_{n+i})_{i\in[n]}
    \succeq0.
  \end{equation}

  For each $j\in[k_1]$, define
  \begin{equation}\label{eq:halfline-Fj}
    F_j
    \coloneqq
    \widehat{K}_{0,1}^{\mathsf T}
    \mathcal{H}_0^\dagger
    \widehat{K}_{0,1}
    +
    \sum_{i=1}^{2n+1}
      \binom{2n+1}{i}
      (-1)^it_j^i
      \widehat{S}_{2n+1-i,1}
    -
    \widehat{K}_{t_j,1}^{\mathsf T}
    \mathcal{H}_{t_j}^\dagger
    \widehat{K}_{t_j,1},
  \end{equation}
  where $\widehat{S}_{i,1}$ denotes the upper-left
  $p_0\times p_0$ principal submatrix of
  \[
    \widehat{S}_i
    \coloneqq
    P^{\mathsf T}S_iP.
  \]
  The sum in~\eqref{eq:halfline-Fj} is the change-of-basis correction
  from the basis centered at $0$ to the basis centered at $t_j$.

  Since the family $(F_j)_{j\in[k_1]}$ is finite, we may choose
  $\alpha>0$ sufficiently large that
  \[
    \alpha I_{p_0}+F_j\succ0,
    \qquad j\in[k_1].
  \]
  Set
  \begin{equation}\label{eq:halfline-Pplus}
    P_+\coloneqq\alpha I_{p_0}.
  \end{equation}
  In particular, $P_+\succ0$. Define
  \begin{equation}\label{eq:halfline-Z01}
    \widehat{Z}_{0,1}
    \coloneqq
    \widehat{K}_{0,1}^{\mathsf T}
    \mathcal{H}_0^\dagger
    \widehat{K}_{0,1}
    +
    P_+.
  \end{equation}
  As in the proof of Theorem~\ref{th:mainTheorem_RR}, define
  \begin{equation}\label{eq:halfline-Z03-Z06}
    \widehat{Z}_{0,3}
    \coloneqq
    \begin{pmatrix}
      \widehat{K}_{0,1}^{\mathsf T} & \widehat{Z}_{0,1}
    \end{pmatrix}
    \widehat{G},
    \qquad
    \widehat{Z}_{0,6}
    \coloneqq
    \begin{pmatrix}
      \widehat{K}_{0,3}^{\mathsf T}
      &
      \widehat{Z}_{0,3}^{\mathsf T}
    \end{pmatrix}
    \widehat{G},
  \end{equation}
  and set
  \begin{equation}\label{eq:halfline-Z0}
    \widehat{Z}_0
    \coloneqq
    \begin{pmatrix}
      \widehat{Z}_{0,1} & \widehat{Z}_{0,3}\\
      \widehat{Z}_{0,3}^{\mathsf T} & \widehat{Z}_{0,6}
    \end{pmatrix},
    \qquad
    Z_0
    \coloneqq
    P\widehat{Z}_0P^{\mathsf T}.
  \end{equation}
  Since the reference point is $0$, we set
  \[
    S_{2n+1}\coloneqq Z_0.
  \]

  We claim that
  \begin{equation}\label{eq:halfline-permuted-Schur-condition}
    \widehat{Z}_0
    -
    \widehat{K}_0^{\mathsf T}
    \mathcal{H}_0^\dagger
    \widehat{K}_0
    \succeq0.
  \end{equation}

Write
\[
  \widehat G
  =
  \begin{pmatrix}
    \widehat G_0\\
    \widehat G_1
  \end{pmatrix},
  \qquad
  \widehat G_0\in M_{np\times(p-p_0)}(\mathbb R),
  \quad
  \widehat G_1\in M_{p_0\times(p-p_0)}(\mathbb R).
\]
Then
\begin{equation}
\label{eq:halfline-dependent-blocks}
  \widehat K_{0,3}
  =
  \mathcal H_0\widehat G_0
  +
  \widehat K_{0,1}\widehat G_1.
\end{equation}
Moreover, by the definitions of $\widehat Z_{0,3}$ and
$\widehat Z_{0,6}$,
\begin{align}
  \widehat Z_{0,3}
  &=
  \widehat K_{0,1}^{\mathsf T}\widehat G_0
  +
  \widehat Z_{0,1}\widehat G_1,
  \label{eq:halfline-Z03-expansion}\\
  \widehat Z_{0,6}
  &=
  \widehat G_0^{\mathsf T}\mathcal H_0\widehat G_0
  +
  \widehat G_1^{\mathsf T}
    \widehat K_{0,1}^{\mathsf T}\widehat G_0
  +
  \widehat G_0^{\mathsf T}
    \widehat K_{0,1}\widehat G_1
  +
  \widehat G_1^{\mathsf T}
    \widehat Z_{0,1}\widehat G_1.
  \label{eq:halfline-Z06-expansion}
\end{align}
Recall the definition of $\widehat Z_{0,1}$ from \eqref{eq:halfline-Z01}.
Since
\[
  \mathcal C(\widehat K_{0,1})
  \subseteq
  \mathcal C(\mathcal H_0),
\]
and $\mathcal H_0$ is symmetric, we have
\begin{equation}
\label{eq:halfline-pseudoinverse-identities}
  \mathcal H_0\mathcal H_0^\dagger\widehat K_{0,1}
  =
  \widehat K_{0,1},
  \qquad
  \widehat K_{0,1}^{\mathsf T}
  \mathcal H_0^\dagger\mathcal H_0
  =
  \widehat K_{0,1}^{\mathsf T}.
\end{equation}

We now compute the four blocks of
\[
  \widehat Z_0
  -
  \widehat K_0^{\mathsf T}
  \mathcal H_0^\dagger
  \widehat K_0,
  \qquad
  \widehat K_0
  =
  \begin{pmatrix}
    \widehat K_{0,1} & \widehat K_{0,3}
  \end{pmatrix}.
\]
For the upper-left block, \eqref{eq:halfline-Z01} gives
\begin{equation}
\label{eq:halfline-Schur-block-11}
  \widehat Z_{0,1}
  -
  \widehat K_{0,1}^{\mathsf T}
  \mathcal H_0^\dagger
  \widehat K_{0,1}
  =
  P_+.
\end{equation}
Using \eqref{eq:halfline-dependent-blocks},
\eqref{eq:halfline-Z03-expansion}, and
\eqref{eq:halfline-pseudoinverse-identities}, we obtain
\begin{align}
  \widehat Z_{0,3}
  -
  \widehat K_{0,1}^{\mathsf T}
  \mathcal H_0^\dagger
  \widehat K_{0,3}
  &=
  \widehat K_{0,1}^{\mathsf T}\widehat G_0
  +
  \widehat Z_{0,1}\widehat G_1
  \notag\\
  &\quad
  -
  \widehat K_{0,1}^{\mathsf T}
  \mathcal H_0^\dagger
  \bigl(
    \mathcal H_0\widehat G_0
    +
    \widehat K_{0,1}\widehat G_1
  \bigr)
  \notag\\
  &=
  \widehat K_{0,1}^{\mathsf T}\widehat G_0
  +
  \widehat Z_{0,1}\widehat G_1
  -
  \widehat K_{0,1}^{\mathsf T}\widehat G_0
  -
  \widehat K_{0,1}^{\mathsf T}
    \mathcal H_0^\dagger
    \widehat K_{0,1}\widehat G_1
  \notag\\
  &=
  \left(
    \widehat Z_{0,1}
    -
    \widehat K_{0,1}^{\mathsf T}
      \mathcal H_0^\dagger
      \widehat K_{0,1}
  \right)\widehat G_1
  \notag\\
  &=
  P_+\widehat G_1.
  \label{eq:halfline-Schur-block-13}
\end{align}
Taking transposes yields
\begin{equation}
\label{eq:halfline-Schur-block-31}
  \widehat Z_{0,3}^{\mathsf T}
  -
  \widehat K_{0,3}^{\mathsf T}
  \mathcal H_0^\dagger
  \widehat K_{0,1}
  =
  \widehat G_1^{\mathsf T}P_+.
\end{equation}

Finally, by \eqref{eq:halfline-dependent-blocks} and
\eqref{eq:halfline-Z06-expansion},
\begin{align}
  &\widehat Z_{0,6}
  -
  \widehat K_{0,3}^{\mathsf T}
  \mathcal H_0^\dagger
  \widehat K_{0,3}
  \notag\\
  &=
  \widehat G_0^{\mathsf T}\mathcal H_0\widehat G_0
  +
  \widehat G_1^{\mathsf T}
    \widehat K_{0,1}^{\mathsf T}\widehat G_0
  +
  \widehat G_0^{\mathsf T}
    \widehat K_{0,1}\widehat G_1
  +
  \widehat G_1^{\mathsf T}
    \widehat Z_{0,1}\widehat G_1
  \notag\\
  &\quad
  -
  \bigl(
    \widehat G_0^{\mathsf T}\mathcal H_0
    +
    \widehat G_1^{\mathsf T}\widehat K_{0,1}^{\mathsf T}
  \bigr)
  \mathcal H_0^\dagger
  \bigl(
    \mathcal H_0\widehat G_0
    +
    \widehat K_{0,1}\widehat G_1
  \bigr)
  \notag\\
  &=
  \widehat G_1^{\mathsf T}
  \left(
    \widehat Z_{0,1}
    -
    \widehat K_{0,1}^{\mathsf T}
      \mathcal H_0^\dagger
      \widehat K_{0,1}
  \right)
  \widehat G_1
  \notag\\
  &=
  \widehat G_1^{\mathsf T}P_+\widehat G_1.
  \label{eq:halfline-Schur-block-33}
\end{align}
Here we used
$
  \mathcal H_0\mathcal H_0^\dagger\mathcal H_0
  =
  \mathcal H_0
$
together with \eqref{eq:halfline-pseudoinverse-identities}.

Combining
\eqref{eq:halfline-Schur-block-11},
\eqref{eq:halfline-Schur-block-13},
\eqref{eq:halfline-Schur-block-31}, and
\eqref{eq:halfline-Schur-block-33}, we obtain
\begin{align}
  \widehat{Z}_0
  -
  \widehat{K}_0^{\mathsf T}
  \mathcal{H}_0^\dagger
  \widehat{K}_0
  &=
  \begin{pmatrix}
    P_+ & P_+\widehat{G}_1\\
    \widehat{G}_1^{\mathsf T}P_+
      & \widehat{G}_1^{\mathsf T}P_+\widehat{G}_1
  \end{pmatrix}
  \notag\\
  &=
  \begin{pmatrix}
    I_{p_0}\\
    \widehat{G}_1^{\mathsf T}
  \end{pmatrix}
  P_+
  \begin{pmatrix}
    I_{p_0} & \widehat{G}_1
  \end{pmatrix}
  \succeq0.
  \label{eq:halfline-Schur-factorization}
\end{align}
  This proves~\eqref{eq:halfline-permuted-Schur-condition}.

  Since
  \[
    \widehat{K}_0=K_0P
    \quad\text{and}\quad
    S_{2n+1}=P\widehat{Z}_0P^{\mathsf T},
  \]
  congruence by $P$ in
  \eqref{eq:halfline-permuted-Schur-condition} yields
  \[
    S_{2n+1}
    -
    K_0^{\mathsf T}
    \mathcal{H}_0^\dagger
    K_0
    \succeq0.
  \]
  Since
  \[
    K_0^{\mathsf T}
    =
    \row(S_{n+i})_{i\in[n]},
    \qquad
    \mathcal{H}_0=\mathcal{H}_x(n-1),
  \]
  condition~\eqref{eq:halfline-Schur-condition} follows.

  We next verify the rank conditions at the prescribed points. For
  each $j\in[k_1]$, define
  \begin{equation}\label{eq:halfline-Ztj}
    \widehat{Z}_{t_j}
    \coloneqq
    \widehat{Z}_0
    +
    \sum_{i=1}^{2n+1}
      \binom{2n+1}{i}
      (-1)^it_j^i
      \widehat{S}_{2n+1-i}.
  \end{equation}
  Let $\widehat{Z}_{t_j,1}$ denote its upper-left
  $p_0\times p_0$ principal submatrix. By
  \eqref{eq:halfline-Fj},
  \eqref{eq:halfline-Z01}, and
  \eqref{eq:halfline-Ztj}, one has
  \[
    \widehat{Z}_{t_j,1}
    =
    \widehat{K}_{t_j,1}^{\mathsf T}
    \mathcal{H}_{t_j}^\dagger
    \widehat{K}_{t_j,1}
    +
    P_+
    +
    F_j.
  \]
  Since $P_++F_j\succ0$, Lemma~\ref{le:lemma-for-increasing-rank-v2}
  gives
  \begin{equation}\label{eq:halfline-rank-small}
    \rank
    \begin{pmatrix}
      \mathcal{H}_{t_j} & \widehat{K}_{t_j,1}\\
      \widehat{K}_{t_j,1}^{\mathsf T}
        & \widehat{Z}_{t_j,1}
    \end{pmatrix}
    =
    \rank
    \begin{pmatrix}
      \mathcal{H}_{t_j}\\
      \widehat{K}_{t_j,1}^{\mathsf T}
    \end{pmatrix}
    +
    p_0,
    \qquad j\in[k_1].
  \end{equation}

  We now complete the extension exactly as in the proof of
  Theorem~\ref{th:mainTheorem_RR}. Choose
  \[
    S_{2n+2}
    \coloneqq
    \row(S_{n+i})_{i\in[n+1]}
    M(n)^\dagger
    \col(S_{n+i})_{i\in[n+1]}.
  \]
  Define
  \[
    \mathcal{S}^{(2n+2)}
    \coloneqq
    (S_0,S_1,\ldots,S_{2n+2}).
  \]
  The corresponding moment matrix $M(n+1)$ satisfies
  \begin{equation}\label{eq:halfline-flat-extension}
    M(n+1)\succeq0,
    \qquad
    \rank M(n+1)=\rank M(n).
  \end{equation}

  Since by assumption $L$ admits an $[0,\infty)$--representing measure,
  Theorem \ref{th:Stieltjes-matricial} implies that
  \[
    \mathcal{H}_x(n-1)\succeq0
  \qquad
  \text{and}
  \qquad
    \mathcal{C}\!\left(
      \col(S_{n+i})_{i\in[n]}
    \right)
    \subseteq
    \mathcal{C}\!\left(\mathcal{H}_x(n-1)\right).
  \]
  Together with~\eqref{eq:halfline-Schur-condition}, the generalized
  Schur complement criterion gives
  \[
    \mathcal{H}_x(n)
    =
    \begin{pmatrix}
      \mathcal{H}_x(n-1)
        & \col(S_{n+i})_{i\in[n]}\\
      \row(S_{n+i})_{i\in[n]}
        & S_{2n+1}
    \end{pmatrix}
    \succeq0.
  \]
  By the rank equality in \eqref{eq:halfline-flat-extension}, the last block column of $M(n+1)$ belongs to the column space of its first $n+1$ block columns. 
  Restricting the corresponding column relation to the last $(n+1)p$ rows of $M(n+1)$, whose first $n+1$ block columns form 
      $\cH_x(n)$, gives
\[
    \mathcal{C}\!\left(
      \col(S_{n+1+i})_{i\in[n+1]}
    \right)
    \subseteq
    \mathcal{C}\!\left(\mathcal{H}_x(n)\right).
  \]
  Therefore, Theorem \ref{th:Stieltjes-matricial}, applied with $n$ replaced by
  $n+1$, shows that the Riesz mapping
  $L_{\mathcal{S}^{(2n+2)}}$ admits a
  $(\rank M(n+1))$--atomic $[0,\infty)$--representing matrix measure
  $\mu$. By~\eqref{eq:halfline-flat-extension},
  \[
    \rank M(n+1)=\rank M(n),
  \]
  and hence $\mu$ is $(\rank M(n))$--atomic. Its restriction to
  $\mathbb{R}[x]_{\leq2n}$ represents the original operator $L$.

  Finally, the change-of-basis and rank arguments from the proof of
  Theorem~\ref{th:mainTheorem_RR} apply without modification. In
  particular, for every $j\in[k_1]$,
  \[
    P\widehat{Z}_{t_j}P^{\mathsf T}
    =
    \mathcal{E}_{t_j}(S_{2n+1}),
  \]
  and
  \begin{align*}
    \rank
    \begin{pmatrix}
      \mathcal{H}_{t_j} & K_{t_j}\\
      K_{t_j}^{\mathsf T}
        & \mathcal{E}_{t_j}(S_{2n+1})
    \end{pmatrix}
    &=
    \rank
    \begin{pmatrix}
      \mathcal{H}_{t_j} & \widehat{K}_{t_j,1}\\
      \widehat{K}_{t_j,1}^{\mathsf T}
        & \widehat{Z}_{t_j,1}
    \end{pmatrix}\\
    &=
    \rank
    \begin{pmatrix}
      \mathcal{H}_{t_j}\\
      K_{t_j}^{\mathsf T}
    \end{pmatrix}
    +
    p-\Card\Dep\!\left(X_{\mathcal{T}_{t_j}}^n\right),
  \end{align*}
  where the second equality follows from
  \eqref{eq:halfline-rank-small} and the definition of $p_0$.

  Corollary~\ref{co:theorem-subsequent-moments-corollary-m=0} now
  yields
  \[
    \mult_\mu t_j
    =
    \Card A_{\mathcal{T}_{t_j}},
    \qquad j\in[k_1].
  \]
  This completes the proof.
\end{proof}

The following corollary characterizes when the representing measure in
Theorem~\ref{th:mainTheorem_halfline} can be chosen so that none of the
prescribed points is an atom.

\begin{corollary}
  \label{co:corollary_mainTheorem_RR_halfline}
  Let $n,p\in\mathbb{N}$, and let
  $t_1,\ldots,t_{k_1}\in [0,\infty)$ be pairwise distinct. Let
  \[
    L\colon\mathbb{R}[x]_{\leq 2n}\to\Sym_p(\mathbb{R})
  \]
  be a linear operator with an $[0,\infty)$--representing matrix
  measure. 
  Then the following
  statements are equivalent:
  \begin{enumerate}
    \item\label{co:corollary_mainTheorem_RR_halfline-pt0}
    There exists a finitely atomic
    $[0,\infty)$--representing matrix measure $\mu$ for $L$ such that
    \[
      t_j\notin\supp(\mu),
      \qquad j\in[k_1].
    \]
    
    \item\label{co:corollary_mainTheorem_RR_halfline-pt1}
    There exists a $(\rank M(n))$--atomic
    $[0,\infty)$--representing matrix measure $\mu$ for $L$ such that
    \[
      t_j\notin\supp(\mu),
      \qquad j\in[k_1].
    \]

    \item\label{co:corollary_mainTheorem_RR_halfline-pt2}
    For every $j\in[k_1]$, we have
    \[
      A_{\mathcal{T}_{t_j}}=\emptyset,
    \]
    where $A_{\mathcal{T}_{t_j}}$ is defined
    in~\eqref{def:def-of-A_t}.
  \end{enumerate}
\end{corollary}

\begin{proof}
  Assume first that
  \eqref{co:corollary_mainTheorem_RR_halfline-pt0} holds. 
   Then
  $\mult_\mu t_j=0$,
    $j\in[k_1].$
  It follows from Corollary \ref{co:theorem-subsequent-moments-finitely-atomic} that
  $A_{\mathcal{T}_{t_j}}=\emptyset$ for every $j\in[k_1]$. Hence,
  \eqref{co:corollary_mainTheorem_RR_halfline-pt2} holds.
  
  Assume that
  \eqref{co:corollary_mainTheorem_RR_halfline-pt2} holds. By
  Theorem~\ref{th:mainTheorem_halfline}, there exists a
  $(\rank M(n))$--atomic $[0,\infty)$--representing matrix measure
  $\mu$ for $L$ such that
  $
    \mult_\mu t_j
    =
    \Card A_{\mathcal{T}_{t_j}}
    =
    0,
$
    $j\in[k_1].$
  Therefore,
  $t_j\notin\supp(\mu)$, $j\in[k_1]$,
  which proves
  \eqref{co:corollary_mainTheorem_RR_halfline-pt1}.

The implication
  $\eqref{co:corollary_mainTheorem_RR_halfline-pt1}\Rightarrow\eqref{co:corollary_mainTheorem_RR_halfline-pt0}$ is clear.
\end{proof}


We are now ready to prove Theorem~\ref{thm:main-1} in the case
$K=[0,\infty)$.

\begin{proof}[Proof of Theorem~\ref{thm:main-1} for $K=[0,\infty)$]
    The argument is identical to that for $K=\RR$ in
    Section~\ref{sec:solution-to-the-K-MRTMP}, except that
    Theorem~\ref{th:Hamburger-matricial} and
    Corollary~\ref{co:corollary_mainTheorem_RR}
    are replaced by Theorem~\ref{th:Stieltjes-matricial} and
    Corollary~\ref{co:corollary_mainTheorem_RR_halfline}, respectively.
\end{proof}


\section{Example for Theorem \ref{thm:main-1}}
\label{sec:example}

We illustrate the finite linear-algebraic test in Theorem~\ref{thm:main-1}. We construct a flat extension
and recover the representing measure from it.  The example also shows that
the pole-avoidance condition contains information not captured by the
ordinary matricial Hamburger conditions.

Let \(p=2\), \(K=\mathbb R\), and
$
  q(x)=x^2(x-1)^2.
$
Thus \(e_0=0\), \(e_1=e_2=1\), \(2n=4\) (see Section \ref{introduction}), and the real poles of $q$ are $0$ and
\(1\).  Define \(\mathcal L\colon\mathcal R^{(4)}\to\Sym_2(\mathbb R)\)
by prescribing the moments of its associated polynomial mapping
\(L(g):=\mathcal L(g/q)\):
\[
S_0=\begin{pmatrix}4&1\\1&1\end{pmatrix},\quad
S_1=\begin{pmatrix}2&3\\3&3\end{pmatrix},\quad
S_2=\begin{pmatrix}18&9\\9&9\end{pmatrix},\quad
S_3=\begin{pmatrix}26&27\\27&27\end{pmatrix},\quad
S_4=\begin{pmatrix}114&81\\81&81\end{pmatrix}.
\]
The corresponding moment matrix is
\[
M(2)=
\begin{pmatrix}
4&1&2&3&18&9\\
1&1&3&3&9&9\\
2&3&18&9&26&27\\
3&3&9&9&27&27\\
18&9&26&27&114&81\\
9&9&27&27&81&81
\end{pmatrix}.
\]
It is easy to check that 
\[
 M(2)\succeq0,
 \qquad \operatorname{rank}M(1)=3<4=\operatorname{rank}M(2),\qquad
 \text{and}\qquad
 \cC\!\left(\operatorname{col}(S_3,S_4)\right)
 \subseteq \mathcal C(M(1)).
\] 
Thus, \eqref{thm:main-1-iv-a} and \eqref{thm:main-1-iv-b} of Theorem \ref{thm:main-1} are satisfied.
It remains to check \eqref{thm:main-1-iv-c}.   Let \(M_{\mathcal T_t}(2)\) be the moment matrix
written in the shifted basis \(1,x-t,(x-t)^2\), and let
\(y^{(t)}_1,\ldots,y^{(t)}_6\) denote its columns in their natural block
order.  Set
\[
 r_j(t):=\operatorname{rank}\bigl(y^{(t)}_1,\ldots,y^{(t)}_j\bigr)
 \quad (1\leq j\leq6)
\]
and
\[
 s_j(t):=\operatorname{rank}\bigl(y^{(t)}_3,\ldots,y^{(t)}_j\bigr)
 \quad (3\leq j\leq6).
\]
The shifted moments
at \(t=1\) are
\[
T_0=\begin{pmatrix}4&1\\1&1\end{pmatrix},\quad
T_1=\begin{pmatrix}-2&2\\2&2\end{pmatrix},\quad
T_2=\begin{pmatrix}18&4\\4&4\end{pmatrix},\quad
T_3=\begin{pmatrix}-26&8\\8&8\end{pmatrix},\quad
T_4=\begin{pmatrix}114&16\\16&16\end{pmatrix}.
\]
Exact column elimination gives
\[
\begin{array}{c|c|c|c|c|c}
t &(r_1,\ldots,r_6)&(s_3,\ldots,s_6)
&\bigl(\zeta_{\cT_t}(1),\zeta_{\cT_t}(2)\bigr)
&\bigl(\zeta^{(1)}_{\cT_t}(1),\zeta^{(1)}_{\cT_t}(2)\bigr)
&A_{\cT_t}\\ \hline
0 &(1,2,3,3,4,4)&(1,2,3,3)&(\infty,1)&(\infty,2)&\varnothing\\
1 &(1,2,3,3,4,4)&(1,2,3,3)&(\infty,1)&(\infty,2)&\varnothing
\end{array}
\]
in the notation of Subsection~\ref{subsec:column-dependencies}. 
Thus condition~\eqref{thm:main-1-iv-c} of Theorem~\ref{thm:main-1} holds at both poles.  The rational moment
problem is solvable.

\noindent \textbf{A flat extension.}
Set
$S_5=\begin{pmatrix}242&243\\243&243\end{pmatrix},$
$S_6=\begin{pmatrix}858&729\\729&729\end{pmatrix},$
such that 
$\operatorname{rank}M(3)=4=\operatorname{rank}M(2)$.
Thus \(M(3)\) is a flat extension of $M(2)$.
We now recover the measure.
Write \(y_j:=y_j^{(0)}\), now regarding the columns as
columns of \(M(3)\), and group them into block columns
\[
 X^0=(y_1\ y_2),\qquad X^1=(y_3\ y_4),\qquad
 X^2=(y_5\ y_6),\qquad X^3=(y_7\ y_8).
\]
We have that
$$
    y_7=4y_1+20y_2+4y_3-y_5,
    \qquad
    y_8=27y_2.
$$
Equivalently,
\[
 X^3=X^0H_0+X^1H_1+X^2H_2,
\]
where
\[
 H_0=\begin{pmatrix}4&0\\20&27\end{pmatrix},\qquad
 H_1=\begin{pmatrix}4&0\\0&0\end{pmatrix},\qquad
 H_2=\begin{pmatrix}-1&0\\0&0\end{pmatrix}.
\]
Define
\[
 H(x)=-H_0-xH_1-x^2H_2+x^3I_2=
 \begin{pmatrix}
 x^3+x^2-4x-4&0\\
 -20&x^3-27
 \end{pmatrix}.
\]
By the above, 
\(H\) is a block column relation of \(M(3)\).  Moreover,
\[
\begin{aligned}
 \det H(x)
 &=(x^3+x^2-4x-4)(x^3-27)\\
 &=(x+2)(x+1)(x-2)(x-3)(x^2+3x+9).
\end{aligned}
\]
Since \(x^2+3x+9>0\) for every \(x\in\mathbb R\),
\[
 \mathcal Z(\det H)=\{-2,-1,2,3\}.
\]
By Lemma~\ref{lem:support-from-block-column-relation} and $\Rank M(3)=4$, it follows that
$x_1=-2$, $x_2=-1$, $x_3=2$, $x_4=3$ are the atoms of the measure.
Solving the matrix Vandermonde system
\[
 S_k=\sum_{j=1}^4 x_j^kA_j,
 \qquad k=0,1,2,3,
\]
gives
$
 A_{1}=A_{2}=A_3=
 \begin{pmatrix}1&0\\0&0\end{pmatrix},
 $
$
 A_4=\begin{pmatrix}1&1\\1&1\end{pmatrix}.
$
Hence a representing measure is
\[
 \mu=\delta_{-2}A_{1}+\delta_{-1}A_{2}
     +\delta_2A_3+\delta_3A_4.
\]
The support avoids both zeros of \(q\), as predicted by the pole test.  The
representing measure for the original rational data is
\[
 \nu=q\cdot\mu
 =36\delta_{-2}A_{1}+4\delta_{-1}A_{2}
  +4\delta_2A_3+36\delta_3A_4.
\]
Indeed, for every \(g\in\mathbb R[x]_{\leq4}\),
\[
 \int_{\mathbb R}\frac{g}{q}\,d\nu
 =\sum_{j=1}^4g(x_j)A_j
 =L(g)
 =\mathcal L\!\left(\frac{g}{q}\right).
\]
Each mass has rank one, so the total atomic multiplicity is
\(4=\operatorname{rank}M(2)\).



\end{document}